\documentclass{amsart}
\pdfoutput=1

\usepackage{amsmath,amssymb,latexsym}
\usepackage{mathtools}
\usepackage[utf8]{inputenc}
\usepackage{units}
\usepackage{enumitem}
\usepackage{array}

\usepackage{placeins}
\usepackage[obeyspaces,hyphens,spaces]{url}

\usepackage{graphicx,color}
\graphicspath{{img/}}

\usepackage{tikz}
\usetikzlibrary{matrix,arrows,calc}

\usepackage{csquotes}
\usepackage[style=alphabetic,backend=bibtex]{biblatex}
\bibliography{sources}

\usepackage{amsthm}
\usepackage{hyperref}

\makeatletter
\ifcsname phantomsection\endcsname
    \newcommand*{\@gobblenexttocentry}[9]{}
\else
    \newcommand*{\@gobblenexttocentry}[4]{}
\fi
\newcommand*{\addsubsection}{%
    \addtocontents{toc}{\protect\@gobblenexttocentry}%
    \subsection*}
\makeatother

\usepackage{aliascnt}

\newcommand{\mynewtheorem}[4]{
  \if\relax\detokenize{#3}\relax 
    \if\relax\detokenize{#4}\relax 
      \newtheorem{#1}{#2}
    \else
      \newtheorem{#1}{#2}[#4]
    \fi
  \else
    \newaliascnt{#1}{#3}
    \newtheorem{#1}[#1]{#2}
    \aliascntresetthe{#1}
  \fi
  \expandafter\def\csname #1autorefname\endcsname{#2}
}

\def\equationautorefname~#1\null{(#1)\null}

\mynewtheorem{theorem}{Theorem}{}{section}
\mynewtheorem{lemma}{Lemma}{theorem}{}
\mynewtheorem{rem}{Remark}{lemma}{}
\mynewtheorem{prop}{Proposition}{lemma}{}
\mynewtheorem{cor}{Corollary}{lemma}{}

\newcommand*{\Xfamily}{Clover family} 
\newcommand*{\Yfamily}{Windmill family}
\newcommand*{\maxD}{300}

\def\defbb#1{\expandafter\def\csname#1\endcsname{\mathbb{#1}}}
\def\defcal#1{\expandafter\def\csname#1\endcsname{\mathcal{#1}}}
\def\deffrak#1{\expandafter\def\csname#1\endcsname{\mathfrak{#1}}}
\def\defop#1{\expandafter\def\csname#1\endcsname{\operatorname{#1}}}

\makeatletter
\def\defcals#1{\@defcals#1\@nil}
\def\@defcals#1{\ifx#1\@nil\else\defcal{#1}\expandafter\@defcals\fi}
\def\deffraks#1{\@deffraks#1\@nil}
\def\@deffraks#1{\ifx#1\@nil\else\deffrak{#1}\expandafter\@deffraks\fi}
\def\defbbs#1{\@defbbs#1\@nil}
\def\@defbbs#1{\ifx#1\@nil\else\defbb{#1}\expandafter\@defbbs\fi}
\def\defops#1{\@defops#1,\@nil}
\def\@defops#1,#2\@nil{\if\relax#1\relax\else\defop{#1}\fi\if\relax#2\relax\else\expandafter\@defops#2\@nil\fi}
\makeatother

\defbbs{CPZRQD}
\newcommand{\HH}{\mathbb{H}}

\defcals{MAEOXY}
\newcommand{\CC}{\mathcal{C}}

\newcommand{\Hcal}{\mathcal{H}}
\newcommand{\ZZ}{\mathcal{Z}}

\defops{End,Aut,Aff,Jac,SL,PSL,SO,Stab,PStab,Re,Im,div,Gal,mod}

\newcommand{\PrymX}[1][]{\mathcal{P}\PiX_{#1}}
\newcommand{\PrymY}[1][]{\mathcal{P}\PiY_{#1}}
\newcommand{\EllX}[1][]{\mathcal{E}\PiX_{#1}}
\newcommand{\EllY}[1][]{\mathcal{E}\PiY_{#1}}
\newcommand{\EX}{E^{\X}}
\newcommand{\EY}{E^{\Y}}

\makeatletter%
\protected\def\W{\@ifnextchar[{\W@arg}{W_D}}
\def\W@arg[#1]{%
  \def\@D{D}%
  \def\c@W##1,##2,##3\@nil{\def\c@D{##1}\def\n@D{##2}}\c@W#1,,\@nil%
  \def\@W##1=##2=##3\@nil{\def\@DD{##1}\def\@DDD{##2}}\expandafter\@W\c@D==\@nil%
  \ifx\@D\@DD W_{\@DDD}%
    \expandafter\ifx\n@D\relax\relax\else(\n@D)\fi%
  \else W_D(#1)\fi%
}
\protected\def\parexp#1{\@ifnextchar^{(#1)}{#1}}
\protected\def\parind#1{\@ifnextchar_{(#1)}{#1}}
\makeatother%

\newcommand{\alphaX}{\parexp{\alpha^\X}}
\newcommand{\alphaY}{\parexp{\alpha^\Y}}
\newcommand{\alphaZ}{\parexp{\alpha^\ZZ}}
\newcommand{\betaX}{\parexp{\beta^\X}}
\newcommand{\betaY}{\parexp{\beta^\Y}}
\newcommand{\rhoX}{\parexp{\rho^\X}}
\newcommand{\rhoY}{\parexp{\rho^\Y}}
\newcommand{\PrymVar}{\mathcal{P}(X,\rho)}
\newcommand{\PX}{\mathcal{P}(\X_t)}
\newcommand{\PY}{\mathcal{P}(\Y_t)}
\newcommand{\omegaX}{\parexp{\omega^\X}}
\newcommand{\omegaY}{\parexp{\omega^\Y}}
\newcommand{\PiX}{\Pi^\X}
\newcommand{\PiY}{\Pi^\Y}
\newcommand{\FX}{F^\X}
\newcommand{\GX}{G^\X}
\newcommand{\FY}{F^\Y}
\newcommand{\GY}{G^\Y}
\newcommand{\dF}[1]{F^{#1}}
\newcommand{\dG}[1]{G^{#1}}

\newcommand{\Sym}{\mathfrak{S}}

\renewcommand{\i}{\mathrm{i}}
\renewcommand{\d}[1]{\mathrm{d}#1}
\renewcommand{\H}
{\mathrm{H}}

\def\abs#1{\lvert#1\rvert}

\begin{document}

\title{Orbifold points on Prym-Teichmüller curves in genus three}

\author{David Torres-Teigell}
\thanks{The first-named author was supported by the Alexander von Humboldt Foundation.}
\address{FB 12 -- Institut für Mathematik\\Johann Wolfgang Goethe-Universität\\Robert-Mayer-Str. 6--8\\D-60325 Frankfurt am Main}
\curraddr{Departamento de Matem\'{a}ticas, Universidad Aut\'{o}noma de Madrid\\28049 Madrid, Spain}
\email{david.torres@uam.es}

\author{Jonathan Zachhuber}
\thanks{The second-named author was partially supported by ERC-StG 257137.}
\address{FB 12 -- Institut für Mathematik\\Johann Wolfgang Goethe-Universität\\Robert-Mayer-Str. 6--8\\D-60325 Frankfurt am Main}
\email{zachhuber@math.uni-frankfurt.de}

\begin{abstract}
  Prym-Teichmüller curves $\W[4]$ constitute the main examples of known primitive Teichmüller curves in the moduli space $\M_3$. 
  We determine, for each non-square discriminant $D>1$, the number and type of orbifold points in $\W[4]$. These results, together with the formulas of Lanneau-Nguyen and Möller for the number of cusps and the Euler characteristic, complete the topological characterisation of Prym-Teichmüller curves in genus 3.

  Crucial for the determination of the orbifold points is the analysis of families of genus 3 cyclic covers of degree $4$ and $6$, branched over four points of $\P^1$. As a side product of our study, we provide an explicit description of the Jacobians and the Prym-Torelli images of these two families, together with a description of the corresponding flat surfaces.
\end{abstract}

\maketitle

\tableofcontents

\section{Introduction}\label{sec:intro}

A \emph{Teichmüller curve} is an algebraic curve in the moduli space $\M_g$ of genus $g$ curves that is totally geodesic for the Teichmüller metric. Teichmüller curves arise naturally from \emph{flat surfaces}, i.e.\ elements $(X,\omega)$ of the bundle $\Omega\M_g$ over $\M_g$, consisting of a curve $X$ with a holomorphic $1$-form $\omega\in\Omega(X)$. The bundle $\Omega\M_g$ is endowed with an $\SL_2(\R)$-action, defined by affine shearing of the flat structure induced by the differential. In the rare case that the closure of the projection to $\M_g$ of the $\SL_2(\R)$-orbit of an element $(X,\omega)$ is an algebraic curve, i.e. that $(X,\omega)$ has many real symmetries, we obtain a Teichmüller curve.

Only few examples of families of (primitive) Teichmüller curves are known, see \cite{mcmSL2R}, \cite{mcmprym}, \cite{kenyonsmillie} and \cite{bouwmoeller}. 
In genus 2, McMullen was able to construct the \emph{Weierstraß curves}, and thereby classify all Teichmüller curves in $\M_2$ by analysing when the Jacobian of the flat surface admits real multiplication that respects the $1$-form. However, for larger genus, requiring real multiplication on the entire Jacobian is too strong a restriction. By relaxing this condition he constructed the \emph{Prym-Teichmüller curves} $\W[4]$ in genus $3$ and $\W[6]$ in genus $4$ (see \autoref{sec:orbifoldbg} for definitions). Recent results suggest that almost all Teichmüller curves in genus $3$ are of this form, see \cite{ANW}, \cite{matheuswright}, \cite{nguyenwright} and \cite{aulicinonguyen}.

While the situation for genus $2$ is fairly well understood, things are less clear for higher genus. As curves in $\M_g$, Teichmüller curves carry a natural orbifold structure. As such, one is primarily interested in their homeomorphism type, i.e. the genus, the number of cusps, components, and the number and type of orbifold points. In genus two, this was solved for the Weierstraß curves by McMullen~\cite{mcmTCingenustwo}, Bainbridge~\cite{bainbridgeeulerchar} and Mukamel~\cite{mukamelorbifold}.

For the Prym-Teichmüller curves in genus $3$ and $4$, the Euler characteristics were calculated by Möller~\cite{moellerprym} and the number of components and cusps were counted by Lanneau and Nguyen~\cite{lanneaunguyen}. The primary aim of this paper is to describe the number and type of orbifold points occurring in genus $3$, thus completing the topological characterisation of $\W[4]$ for all (non-square) discriminants $D$ via the formula
    \begin{equation}\label{eq:invariants}
    2h_0-2g= \chi+C+\sum_{d} e_{d}\left(1-\frac{1}{d}\right)
    \end{equation}
where $g$ denotes the genus of $\W[4]$, $h_0$ the number of components, $\chi$ the Euler characteristic, $C$ the number of cusps and $e_{d}$ the number of orbifold points of order $d$. As $\W[4]$ is either connected or the connected components are homeomorphic by \cite{components}, this characterises all Teichmüller curves inside the loci $\W[4]$.

Except for some extra symmetries occurring for small $D$, we describe the orbifold points in terms of integral solutions of ternary quadratic forms, which lie in some fundamental domain. 
More precisely, for any positive discriminant $D$, we define
\begin{align*}
\Hcal_{2}(D)\coloneqq \{(a,b,c) \in \Z^{3} : & \  a^{2}+b^{2}+c^{2}=D\ ,\ \gcd(a,b,c,f_{0})=1\,\},\text{ and} \\[0.2cm]
\Hcal_{3}(D)\coloneqq \{(a,b,c)\in\Z^{3} : & \
      2a^{2}-3b^{2}-c^{2}=2D\ ,\ \gcd(a,b,c,f_{0})=1\ , \\
      &\nonumber \ -3\sqrt{D}<a<-\sqrt{D}\ ,\ c < b \le 0\ , \\ 
      &\nonumber \ (4a - 3b - 3c <0)\vee(4a - 3b - 3c=0\ \wedge\ c < 3b)\,\}, 
\end{align*}
where $f_{0}$ denotes the conductor of $D$. The extra conditions in the definition of $\Hcal_{3}(D)$ restrict the solutions to a certain fundamental domain. In particular, even though the quadratic form is indefinite, these conditions ensure that the set $\Hcal_3(D)$ is finite for all $D$. 

\begin{theorem}\label{thm:mainthmshort}
For non-square discriminant~$D>12$, the Prym-Teichmüller curves $\W[4]$ for genus three have orbifold points of order $2$ or $3$.

More precisely, the number $e_{3}(D)$ of orbifold points of order 3 is $\abs{\Hcal_3(D)}$; the number $e_{2}(D)$ of orbifold points of order 2 is $\abs{\Hcal_2(D)}/24$ if $D$ is even and there are no points of order 2 when $D$ is odd.


%
%
%

The curve $\W[D=8](4)$ has one point of order $3$ and one point of order $4$; the curve $\W[D=12](4)$ has a single orbifold point of order $6$.
\end{theorem}

Let us recall that $\W[4]$ is empty for $D\equiv 5\mod 8$ (see~\cite[Prop. 1.1]{moellerprym}).





\autoref{thm:mainthmshort} combines the content of \autoref{thm:orbpt2and4} and \autoref{thm:orbpt3and6}.
The topological invariants of $\W[4]$ for $D$ up to \maxD{} are given in \autoref{tab:thetable} on page~\pageref{tab:thetable}.

Our approach to solving this problem is purely algebraic and therefore the use of tools from the theory of flat surfaces will be sporadic.

Two families of curves will play a special role in determining orbifold points on Prym-Teichmüller curves, namely the \emph{\Xfamily{}} and the \emph{\Yfamily{}}, which will be introduced in~\autoref{sec:cycliccovers}. They parametrise certain genus 3 cyclic covers of $\P^1$ of degree 4 and 6, respectively. There are two special points in these families, namely the \emph{Fermat curve} of degree 4, which is the only element of the \Xfamily{} with a cyclic group of automorphisms of order 8, and the exceptional \emph{Wiman curve} of genus 3, which is the unique intersection of the two families and the unique curve in genus $3$ that admits a cyclic group of automorphisms of order 12.

The fact that orbifold points in $\W[4]$ correspond to points of intersection with these two families will follow from the study of the action of the Veech group~$\SL(X,\omega)$ carried out in~\autoref{sec:orbifoldbg}. A consequence of this study is that orbifold points of order 4 and 6 correspond to the Fermat and Wiman curves, respectively, while points of order 2 and 3 correspond to generic intersections with the \Xfamily{} and the \Yfamily{}, respectively.

In order to determine these points of intersection, we will need a very precise description of the two families or, more precisely, of their images under the Prym-Torelli map. To this end, we explicitly compute the period matrices of the two families in \autoref{sec:Prym}. While the analysis of different types of orbifold points was rather uniform up to this point, the \Xfamily{} and the \Yfamily{} behave quite differently under the Prym-Torelli map. In particular, the Prym-Torelli image of the \Xfamily{} is constant.

\begin{theorem}\label{thm:Xfamily}The Prym-Torelli image of the \Xfamily{} $\X$ is isogenous to the point $E_{\i}\times E_{\i}$ in the moduli space $\A_{2,(1,2)}$ of abelian surfaces with $(1,2)$-polarisation, where $E_{\i}$ denotes the elliptic curve corresponding to the square torus $\C/(\Z\oplus\Z\i)$. 
Orbifold points on $\W[4]$ of order 2 and 4 correspond to intersections with this family.
\end{theorem}

In contrast, the image of the \Yfamily{} under the Prym-Torelli map lies in the Shimura curve of discriminant 6. We show this by giving a precise description of the endomorphism ring of a general member of this family (see~\autoref{prop:shimura}).

\begin{theorem}\label{thm:Yfamily}The closure of the Prym-Torelli image of the \Yfamily{} $\Y$ in $\A_{2,(1,2)}$ is the (compact) Shimura curve parametrising $(1,2)$-polarised abelian surfaces with endomorphism ring isomorphic to the maximal order in the indefinite rational quaternion algebra of discriminant 6. Orbifold points of $\W[4]$ of order 3 and 6 correspond to intersections with this family.
\end{theorem}

The relationship between the \Xfamily{}, the \Yfamily{}, and a Prym-Teichmüller curve is illustrated in \autoref{fig:curveintersection}.
\begin{figure}
\centering
\includegraphics{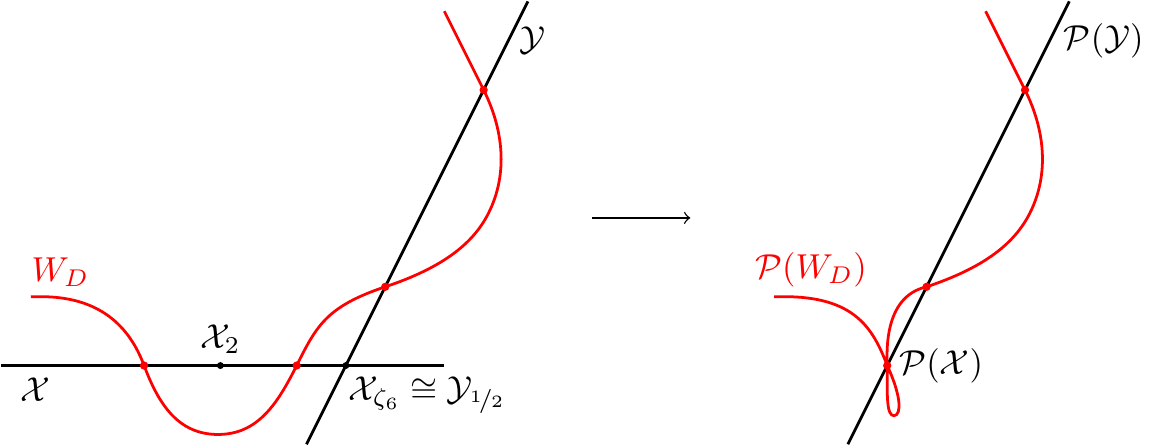}
\caption{The \Xfamily{}, the \Yfamily{}, and the curve $\W$ inside $\M_3$ and their image under the Prym-Torelli map in $\A_{2,(1,2)}$.}
\label{fig:curveintersection}
\end{figure}

In~\autoref{sec:orbifoldpts}, we finally determine the intersections of the Prym-Teichmüller curve $\W[4]$ with the \Xfamily{} and the \Yfamily{} by studying which points in their Prym-Torelli images admit real multiplication by the quadratic order $\O_{D}$ and by determining the corresponding eigenforms for this action. An immediate consequence is the following result.

\begin{cor}\label{cor:order4and6}The only Prym-Teichmüller curves in $\M_3$ with orbifold points of order 4 or 6 are $\W[D=8,4]$ of genus zero with one cusp, one point of order 3 and one point of order 4, and $\W[D=12,4]$ of genus zero with two cusps and one point of order~6.
\end{cor}

Note that our result extends that of Mukamel in~\cite{mukamelorbifold} to genus $3$, although our approach and techniques differ in almost every detail. In the following we give a brief summary of the techniques used to classify orbifold points of Weierstraß curves in genus~$2$, to illustrate the similarities with and differences to our case.

The first difference is that, while in genus $2$ all curves are hyperelliptic, this is never the case for genus $3$ curves on Prym-Teichmüller curves by \autoref{nonhyp}. Luckily, the Prym involution is a satisfactory substitute in all essential aspects. In particular, while Mukamel obtains restrictions on the types of orbifold points in genus~$2$ by observing the action on the Weierstraß points, we acquire an analogous result in genus~$3$ by relating symmetries of Prym forms to automorphisms of elliptic curves (\autoref{thm:orbifoldpts}).

At this point, however, the similarities between the genus $2$ and $3$ cases seem to end.
Mukamel shows that the orbifold points on genus~$2$ Weierstraß curves correspond to curves admitting an embedding of the dihedral group $D_8$ into their automorphism group and whose Jacobians are therefore isogenous to products of elliptic curves that admit complex multiplication.
He then identifies the space of genus~$2$ curves admitting a faithful $D_8$ action with the modular curve $\HH/\Gamma_0(2)$. In this model, the curves admitting complex multiplication are well-known to correspond to the imaginary quadratic points in the fundamental domain. Thus counting orbifold points in genus $2$ is equivalent to computing class numbers of imaginary quadratic fields, as in the case of Hilbert Modular Surfaces. Moreover, this period domain permits associating concrete flat surfaces to the orbifold points via his \enquote{pinwheel} construction.

By contrast, in genus~$3$, each orbifold point may lie on the \Xfamily{} \textit{or} the \Yfamily{} (\autoref{thm:families}). As mentioned above, these two cases behave quite differently. Moreover, in genus~$3$ we are no longer dealing with the entire Jacobian, but only with the Prym part, i.e. part of the Jacobian collapses and the remainder carries a non-principal~$(1,2)$ polarisation (see \autoref{sec:orbifoldbg}).
In particular, while in Mukamel's case the appearing abelian varieties could all be obtained by taking products of elliptic curves, in genus~3 one is forced to construct the Jacobians \enquote{from scratch} via Bolza's method (\autoref{sec:Prym}). In addition, the whole \Xfamily{} collapses to a single point under the Prym-Torelli map, making it more difficult to keep track of the differentials.
All this adds a degree of difficulty to pinpointing the actual intersection points of the \Xfamily{} and the \Yfamily{} with a given $\W[4]$. One consequence is that we obtain class numbers determining the number of orbifold points that are associated to slightly more involved quadratic forms (\autoref{sec:orbifoldpts}).

Finally, we provide flat pictures of the orbifold points of order $4$ and $6$ in \autoref{sec:flat}.

\addsubsection{Acknowledgements}
We are very grateful to Martin Möller not only for suggesting this project to us, but also for continuous support and patient answering of questions.
Additionally, we would like to thank Jakob Stix, André Kappes and Quentin Gendron for many helpful discussions, and Ronen Mukamel for sharing the computer code with the implementation of his algorithm in~\cite{mukamelfundamental} that we used in the last section.
We also thank~\cite{PARI2} and~\cite{sage} for computational help.


\section{Orbifold points on Prym-Teichmüller curves}\label{sec:orbifoldbg}

The aim of this section is to prove the following statement.

\begin{prop}\label{thm:orbifoldpts}
A flat surface $(X,\omega)$ parametrised by a point in $\W[4]$ is an orbifold point of order $n$ if and only if there exists $\sigma\in\Aut(X)$ of order $2n$ satisfying $\sigma^*\omega=\zeta_{2n}\omega$, where $\zeta_{2n}$ is some primitive order $2n$ root of unity.

The different possibilities are listed in~\autoref{tab:orders}.
\end{prop}
\begin{table}[ht]
\begin{tabular}{l c l}
& $\mathrm{ord}(\sigma)$ & Branching data \\\hline
(i) & 4 & $(0;4,4,4,4)$ \\
(ii) & 6 & $(0;2,3,3,6)$ \\
(iii) & 8 & $(0;4,8,8)$ \\
(iv) & 12 & $(0;3,4,12)$ \\
\end{tabular}
\caption{Possible orders of $\sigma$ and their corresponding branching data.}\label{tab:orders}
\end{table}

Before proceeding with the proof, we briefly recall some notation and background information.

\addsubsection{Orbifold Points}
If $G$ is a finite group acting on a Riemann surface $X$ of genus $g\ge 2$, we define the \textit{branching data} (or signature of the action) as the signature of the orbifold quotient $X/G$, that is $\Sigma \coloneqq (\gamma;m_{1},\ldots,m_{r})$, where $\gamma$ is the genus of the quotient $X/G$ and the projection is branched over $r$ points with multiplicities~$m_{i}$.

Recall that an \emph{orbifold point} of an orbifold $\HH/\Gamma$ is the projection of a fixed point of the action of $\Gamma$, i.e. a point $s\in\HH$ so that
$\Stab_\Gamma(s)=\{A\in\Gamma : A\cdot s = s\}$
is strictly larger than the kernel of the action of $\Gamma$. Observe that this is equivalent to requiring the image of $\Stab_\Gamma(s)$ in $\PSL_2(\R)=\Aut(\HH)$, which we denote by $\PStab_\Gamma(s)$, to be non-trivial. We call the cardinality of $\PStab_\Gamma(s)$ the \emph{(orbifold) order} of~$s$.

In the case of a Teichmüller curve, the close relationship between the uniformising group $\Gamma$ and the affine structure of the fibres permits a characterisation of orbifold points in terms of flat geometry. To make this precise, we need some more notation.

\addsubsection{Teichmüller curves}
Recall that a \emph{flat surface} $(X,\omega)$ consists of a curve $X$ together with a non-zero holomorphic differential form $\omega$ on $X$, which induces a flat structure by integration. Hence we may consider the moduli space of flat surfaces $\Omega\M_g$ as a bundle over the moduli space of genus $g$ curves $\M_g$.
Recall that there is a natural $\SL_2(\R)$ action on $\Omega\M_g$ by shearing the flat structure, which respects -- in particular -- the zeros of the differentials. Every Teichmüller curve arises as the projection to $\M_g$ of the (closed) $\SL_2(\R)$ orbit of some $(X,\omega)$. As $\SO(2)$ acts holomorphically on the fibres, we obtain the following commutative diagram
\[\begin{tikzpicture}
\matrix (m) [matrix of math nodes, row sep=1.5em, column sep=2.5em, text height=1.5ex, text depth=0.25ex]
{\SL_2(\R) & \Omega \M_g \\
 \HH \cong \SO(2)\backslash\SL_2(\R) &\P\Omega\M_g \\
\CC=\HH/\Gamma&\M_g \\};
\path[->,font=\scriptsize]
(m-1-1) edge (m-2-1)
(m-1-1) edge [above] node {$F$} (m-1-2)
(m-2-1) edge [above] node {$f$} (m-2-2)
(m-2-1) edge [below] node {} (m-3-1)
(m-3-1) edge (m-3-2)
(m-1-2) edge (m-2-2)
(m-2-2) edge [right] node {$\pi$} (m-3-2);
\end{tikzpicture}\]
where the map $F$ is given by the action $A\mapsto A\cdot(X,\omega)$ and $\CC$ is uniformised by
\[\Gamma=\Stab(f)\coloneqq\{A\in\SL_2(\R) : f(A\cdot t)=f(t)\;,\ \forall t\in\HH\}=\left(\begin{smallmatrix}-1&0\\0&1\\\end{smallmatrix}\right)\cdot\SL(X,\omega)\cdot\left(\begin{smallmatrix}-1&0\\0&1\end{smallmatrix}\right).\]
Here, $\SL(X,\omega)$ is the \emph{affine group} of $(X,\omega)$, i.e. the derivatives of homeomorphisms of $X$ that are affine with regard to the flat structure.

Given $t\in\HH$, we will write $A_t\in\SL_2(\R)$ for (a representative of) the corresponding element in $\SO(2)\backslash\SL_2(\R)$ and $(X_{t},\omega_{t})$ for (a representative of) $f(t)=[A_t\cdot(X,\omega)]\in \P\Omega\M_g$.

For proofs and details, see e.g.~\cite{tmcintro},~\cite{kucharczyk},~\cite{mcmtmchms}.

In the following, we will be primarily interested in a special class of Teichmüller curves.

\addsubsection{Prym-Teichmüller curves}
To ensure that the $\SL_2(\R)$ orbit of a flat surface is not too large, the flat structure must possess sufficient real symmetries. McMullen observed that in many cases this can be achieved by requiring the Jacobian to admit real multiplication that ``stretches'' the differential. However, it turns out that for genus greater than $2$, requiring the whole Jacobian to admit real multiplication is too strong a restriction.

More precisely, for positive $D\equiv 0,1\bmod 4$ non-square, we denote by $\O_D=\Z[T]/(T^2+bT+c)$ with $D=b^2-4c$, the unique (real) quadratic order associated to $D$ and say that a (polarised) abelian surface $A$ has \emph{real multiplication} by $\O_D$ if it admits an embedding $\O_D\hookrightarrow\End(A)$ that is self-adjoint with respect to the polarisation. We call the real multiplication by $\O_D$ \emph{proper}, if the embedding cannot be extended to any quadratic order containing $\O_D$.

Now, consider a curve~$X$ with an involution~$\rho$. The projection $\pi\colon X\rightarrow X/\rho$ induces a morphism $\Jac(\pi)\colon\Jac(X)\rightarrow\Jac(X/\rho)$ of the Jacobians and we call the kernel~$\PrymVar$ of~$\Jac(\pi)$ the \emph{Prym variety} associated to~$(X,\rho)$. In the following, we will always require the Prym variety to be $2$-dimensional, hence the construction only works for $X$ of genus $2$, $3$, $4$ or $5$. Denoting by $\Omega(X)^{+}$ and $\Omega(X)^{-}$ the $+1$ and $-1$-eigenspaces of $\Omega(X)$ with respect to $\rho$, and by $\H_{1}^{+}(X,\Z)$ and $\H_{1}^{-}(X,\Z)$ the corresponding  intersections $\H_{1}(X,\Z)\cap(\Omega(X)^{\pm})^{\vee}$, the Prym variety $\PrymVar$ agrees with $(\Omega(X)^{-})^{\vee}/\H_1^{-}(X,\Z)$. Observe that, when $X$ has genus 3, the Prym variety $\PrymVar$ is no longer principally polarised but carries a $(1,2)$-polarisation. See for instance~\cite[Chap. 12]{birkenhakelange} or~\cite{moellerprym} for details.

Starting with a flat surface $(X,\omega)$ where $X$ admits an involution $\rho$ satisfying $\rho^*\omega=-\omega$ and identifying $\Jac(X)$ with $\Omega(X)^\vee/\H_1(X,\Z)$, the differential $\omega$ is mapped into the Prym part and hence, whenever $\PrymVar$ has real multiplication by $\O_D$, we obtain an induced action of $\O_D$ on $\omega$. We denote by $\E_D(2g-2)\subset\Omega\M_{g}$ the space of $(X,\omega)$ such that
\begin{enumerate}
\item $X$ admits an involution $\rho$ such that $\PrymVar$ is $2$-dimensional,
\item the form $\omega$ has a single zero and satisfies $\rho^*\omega=-\omega$, and
\item $\PrymVar$ admits proper real multiplication by $\O_D$ with $\omega$ as an eigenform,
\end{enumerate}
and by $\P\E_D(2g-2)$ the corresponding quotient by the $\SO(2)$ action.
McMullen showed~\cite{mcmtmchms,mcmprym} that by defining $\W[2g-2]$ as the projection of the locus $\E_D(2g-2)$ to $\M_g$, we obtain (possibly a union of) Teichmüller curves for every discriminant $D$ in $\M_2$, $\M_3$ and $\M_4$. In the genus $2$ case, the Prym involution is given by the hyperelliptic involution and the curve $\W[2]$ is called the \emph{Weierstraß curve}, while the curves $\W[4]$ and $\W[6]$ in $\M_3$ and $\M_4$, respectively, are known as \emph{Prym-Teichmüller curves}. As we are primarily interested in the genus $3$ case, we shall frequently refer to $\W[4]$ simply by $\W$.

We are now in a position to give a precise characterisation of orbifold points on Teichmüller curves in terms of flat geometry.

\begin{prop}\label{prop:orbifoldpoints} Let $\HH/\Gamma$ be a Teichmüller curve generated by some $(X,\omega)=(X_\i,\omega_\i)$.
Then the following are equivalent.
\begin{itemize}
\item The point $t\in \HH$ projects to an orbifold point in $\HH/\Gamma$.
\item There exists an elliptic matrix $C\in\SL(X,\omega)$, $C\neq\pm 1$ such that~$A_{t}CA_{t}^{-1}\in\SO(2)$.
\item The corresponding flat surface $(X_t,\omega_t)$ admits a (holomorphic) automorphism $\sigma$ satisfying $[\sigma^*\omega_t]=[\omega_t]$ and~$\sigma^*\omega_t\neq\pm\omega_t$.
\end{itemize}
\end{prop}

\begin{proof}By the above correspondence, $t\in\HH$ corresponds to some $(X_t,[\omega_t])\in\P\Omega\M_g$ and equivalently to some $A_t\in\SO(2)\backslash\SL_2(\R)$ with $[A_t\cdot(X,\omega)]=(X_t,[\omega_t])$.

Now, $C\in\SL(X,\omega)$ is in the stabiliser of $A_t$ if and only if there exists $B\in\SO(2)$ such that
\[A_{t}C=BA_{t},\text{ i.e. }A_{t}CA_{t}^{-1}\in\SO(2).\]
But then, by definition, $C\in\SL(X,\omega)$ is elliptic. Moreover, $C'\coloneqq A_{t}CA_{t}^{-1}$ lies in $\SL(A_{t}\cdot(X,\omega))=\SL(X_t,\omega_t)$, and as $C'\in\SO(2)$, the associated affine map is in fact a holomorphic automorphism $\sigma$ of $X_t$. In particular, $\sigma^*\omega_t=\zeta\omega_t\in[\omega_t]$, where $\zeta$ is the corresponding root of unity.

Finally, observe that $C$ acts trivially on $\SO(2)\backslash\SL_2(\R)$ if and only if for every $A\in\SL_2(\R)$ there exists $B\in\SO(2)$ so that
\[AC=BA,\text{ i.e. } ACA^{-1}\in\SO(2)\;\forall A\in\SL_2(\R)\]
and this is the case if and only if~$C=\pm 1$.
\end{proof}

\begin{cor}\label{orbimorphisms}
There is a one-to-one correspondence between
\begin{itemize}
\item elements in~$\Stab_\Gamma(t)$,
\item elements in~$\SL\left(A_{t}\cdot(X,\omega)\right)\cap\SO(2)$, and
\item holomorphic automorphisms $\sigma$ of $X_t$ satisfying~$\sigma^*\omega_t\in[\omega_t]$.
\end{itemize}
\end{cor}

In the case of Weierstraß and Prym-Teichmüller curves, we can say even more.

\begin{cor}\label{orbifoldcriterion}
Let $\W[2g-2]$ be as above, let $(X_t,[\omega_t])\in \P\E_D(2g-2)$ correspond to an orbifold point and let $\sigma$ be a non-trivial automorphism of $(X_t,[\omega_t])$. Let $\pi:X_t\to X_t/\sigma$ denote the projection.
Then $\pi$ has a totally ramified point.
\end{cor}

\begin{proof}
As $[\sigma^*\omega]=[\omega]$ and $\omega$ has a single zero, this must be a fixed point of $\sigma$, hence a totally ramified point.
\end{proof}

Note that the Prym-Teichmüller curves $\W[4]$ and $\W[6]$ lie entirely inside the branch locus of $\M_{3}$ and $\M_{4}$ respectively, as all their points admit involutions. In particular, the Prym involution $\rho_{t}$ on each $(X_{t},\omega_{t})$ acts as $-1$, i.e. $\rho_{t}^{*}\omega_{t}=-\omega_{t}$, and therefore it does not give rise to orbifold points.

\begin{cor}\label{orbiindex}
The Prym involution is the only non-trivial generic automorphism of $\W[2g-2]$, i.e. the index $[\Stab_\Gamma(s):\PStab_\Gamma(s)]$ is always~$2$.
\end{cor}

Moreover,~\autoref{prop:orbifoldpoints} gives a strong restriction on the type of automorphisms inducing orbifold points.

\begin{lemma}\label{orbifoldscyclic}
The point in $\W[2g-2]$ corresponding to a flat surface $(X,[\omega])$ is an orbifold point of order $n$ if and only if $(X,[\omega])$ admits an automorphism $\sigma$ of order $2n$. Moreover, $\sigma^n$ is the Prym involution.
\end{lemma}

\begin{proof}
Let $P\in X$ be the (unique) zero of~$\omega$.
By the above, the automorphisms of $(X,[\omega])$ lie in the $P$-stabiliser of $\Aut(X)$. But these are (locally) rotations around $P$, hence the stabiliser is cyclic and of even order, as it contains the Prym involution $\rho$. Conversely, any automorphism $\sigma$ fixing $P$ satisfies $[\sigma^*\omega]=[\omega]$. The remaining claims follow from~\autoref{orbiindex}.
\end{proof}


To determine the number of branch points in the genus $3$ case, we start with the following observation (cf.~\cite[Lemma 2.1]{moellerprym}).

\begin{lemma}\label{nonhyp}
The curve $\W$ is disjoint from the hyperelliptic locus in~$\M_3$.
\end{lemma}

\begin{proof}
Let $(X,[\omega])$ correspond to a point on $\W$, denote by $\rho$ the Prym involution on $X$ and assume that $X$ is hyperelliptic with involution $\sigma$. As $X$ is of genus 3, $\sigma\neq\rho$. But $\sigma$ commutes with $\rho$ and therefore $\tau\coloneqq\sigma\circ\rho$ is another involution.

Recall that $\sigma$ acts by $-1$ on all of $\Omega(X)$ and its  decomposition into $\rho$-eigenspaces $\Omega(X)^{\pm}$. The $-1$ eigenspace of $\tau$ is therefore $\Omega(X)^+$ and the $+1$ eigenspace is $\Omega(X)^-$. In particular, any Prym form on $X$ is $\tau$ invariant, i.e. a pullback from~$X/\tau$.

However, by checking the dimensions of the eigenspaces, we see that $X/\tau$ is of genus $2$, hence $X\rightarrow X/\tau$ is unramified by Riemann-Hurwitz and we cannot obtain a form with a fourfold zero on $X$ by pullback, i.e. $(X,\omega)\not\in\E_D(4)$, a contradiction.
\end{proof}

We now have all we need to prove~\autoref{thm:orbifoldpts}.

\begin{proof}[Proof of~\autoref{thm:orbifoldpts}]
Starting with~\autoref{prop:orbifoldpoints} and~\autoref{orbifoldscyclic}, observe that $\sigma$ descends to an automorphism $\overline{\sigma}$ of the elliptic curve $X/\rho$. Note that $\overline{\sigma}$ acts non-trivially, since $\sigma\neq\rho$, and it has at least one fixed point, hence $X/\sigma\cong\P^1$ and it is well-known that $\overline{\sigma}$ can only be of order $2$, $3$, $4$ or~$6$.

For the number of ramification points, since $X$ has genus $3$, by Riemann-Hurwitz
    \[4=-4n+2n\sum_{d|2n}\left(1-\frac{1}{d}\right)e_{d}\,,\]
where $e_{d}$ is the number of points over which $\sigma$ ramifies with order $d$. A case by case analysis using~\autoref{nonhyp} shows that the only possibilities are those listed in \autoref{tab:orders}.
\end{proof}

\begin{rem}
Automorphism groups of genus 3 curves were classified by Komiya and Kuribayashi in~\cite{KoKu} (P. Henn studied them even earlier in his PhD dissertation~\cite{henn}). One can also find a complete classification of these automorphism groups together with their branching data in~\cite[Table 5]{Bro}, including all the information in our~\autoref{tab:orders}.
\end{rem}


\section{Cyclic covers}\label{sec:cycliccovers}

\autoref{thm:orbifoldpts} classified orbifold points of $\W$ in terms of automorphisms of the complex curve. The aim of this section is to express these conditions as intersections of $\W$ with certain families of cyclic covers of $\P^1$ in~$\M_3$.

Let $\X\rightarrow\P^*\coloneqq\P^1-\{0,1,\infty\}$ be the family of projective curves with affine model
\[\X_t\colon y^4=x(x-1)(x-t)\]
and $\Y\rightarrow\P^*$ the family of projective curves with affine model
\[\Y_t\colon y^6=x^2(x-1)^2(x-t)^3.\]

The family $\X$ has been intensely studied, notably in~\cite{guardia} and~\cite{HS}. In fact, it is even a rare example of a curve that is both a Shimura and a Teichmüller curve (cf. \cite{moellerst}, see~\autoref{rem:wollmilchsau} below). Because of the flat picture of its fibres (cf. \autoref{sec:flat}), we will refer to it as the \Xfamily.

The family $\Y$ is related to the Shimura curve of discriminant 6, which has been studied for instance in~\cite{voight} and~\cite{PetkovaShiga}. We will refer to it as the \Yfamily, again as a reference to the flat picture (cf. \autoref{sec:flat}).

\begin{prop}\label{thm:families}
If $(X,[\omega])$ corresponds to an orbifold point on $\W$ then $X$ is isomorphic to some fibre of $\X$ or~$\Y$.

Moreover, $(X,[\omega])$ is of order six if and only if $X$ is isomorphic to $\X_{\zeta_{6}}\cong\Y_{\nicefrac{1}{2}}$ the (unique) intersection point of $\X$ and $\Y$ in $\M_3$; it is of order four if and only if $X$ is isomorphic to $\X_{-1}$; it is of order two if it corresponds to a generic fibre of $\X$ and of order three if it corresponds to a generic fibre of~$\Y$.
\end{prop}

To state the converse, we need to pick a Prym eigenform on the appropriate fibres of $\X$ and~$\Y$.

\medskip

First, let us briefly review some well-known facts on the theory of cyclic coverings which will be applicable to both the \Xfamily{} $\X$ and the \Yfamily{} $\Y$. For more background and details, see for example~\cite{rohde}.

Consider the family $\ZZ\rightarrow\P^*$ of projective curves with affine model
\[\ZZ_t\colon y^d=x^{a_1}(x-1)^{a_2}(x-t)^{a_3},\]
and choose $a_4$ so that $\sum a_i\equiv 0\bmod d$, with $0<a_{i}<d$. Moreover, we will suppose $\gcd(a_1,a_2,a_3,a_4,d)=1$ so that the curve is connected. Note that any (connected) family of cyclic covers, ramified over four points, may be described in this way.

Let us define $g_{i}=\gcd(a_{i},d)$, for $i=1,\ldots,4$. For each fibre $\ZZ_t$, the map $\pi_t=\pi\colon(x,y)\mapsto x$ yields a cover $\ZZ_t\to\P^{1}$ of degree $d$ ramified over $0$, $1$, $t$ and $\infty$ with branching orders $d/g_{1}$, $d/g_{2}$, $d/g_{3}$ and $d/g_{4}$ respectively. Then, by Riemann-Hurwitz, the genus of $\ZZ_t$ is $d+1-(\sum_{i=1}^4 g_{i})/2$.

Note that the number of preimages of $0$, $1$, $t$ and $\infty$ is $g_{1}$, $g_{2}$, $g_{3}$ and $g_{4}$ respectively. Denote for instance $\pi^{-1}(0)=\{P_{j}\}$, with $j=0,\dotsc,g_{1}-1$. The following map
 \begin{equation}\label{cyclicparametrisation}
 z\mapsto \left(z^{\frac{d}{g_{1}}},\zeta_{d}^{j}z^{\frac{a_1}{g_{1}}}\sqrt[d]{(z^{\frac{d}{g_{1}}}-1)^{a_{2}}(z^{\frac{d}{g_{1}}}-t)^{a_{3}}}\right) \,,\qquad |z|<\varepsilon
 \end{equation}
gives a parametrisation of a neighbourhood of $P_{j}$. In a similar way, one can find local parametrisations around the preimages of the rest of the branching values.

\medskip

The map $\pi$ corresponds to the quotient $\ZZ_{t}/\langle\alphaZ\rangle$ by the action of the cyclic group of order $d$ generated by the automorphism
    \begin{align*}
    \alphaZ\coloneqq \alpha_{t}^\ZZ\colon&(x,y)\mapsto(x,\zeta_{d}y)\,,
    \end{align*}
where $\zeta_{d}=\exp(2\pi\i/d)$. When there is no ambiguity we will simply write $\alpha$ for~$\alphaZ$.

In particular, the cyclic groups acting on $\X_{t}$ and on $\Y_{t}$ are generated by the automorphisms
\begin{align*}
\alphaX\coloneqq \alphaX_{t}\colon&(x,y)\mapsto(x,\zeta_{4}y),
\text{ and }\\
\alphaY\coloneqq \alphaY_{t}\colon&(x,y)\mapsto(x,\zeta_{6}y),
\end{align*}
respectively.

By~\autoref{orbifoldscyclic}, the Prym involutions are given by
\begin{align*}
\rhoX\coloneqq \rho^\X_{t}\coloneqq\alphaX^{2}\colon&(x,y)\mapsto(x,-y)\,,\text{ and }\\
\rhoY\coloneqq \rho^\Y_{t}\coloneqq\alphaY^{3}\colon&(x,y)\mapsto(x,-y)\,.
\end{align*}

We will denote by $\PX$ and $\PY$ the corresponding Prym varieties.

\medskip

Note that different fibres of the families $\X$ and $\Y$ can be isomorphic.

In fact, in the case of the \Xfamily{} $\X$ any isomorphism $\phi:\P^{1}\to\P^{1}$ preserving the set $\{0,1,\infty\}$ lifts to isomorphisms $\X_{t}\cong\X_{\phi(t)}$ for each $t$. As a consequence, our family is parametrised by $\P^{*}/\Sym_{3}$, where we take the symmetric group $\Sym_{3}$ to be generated by $z\mapsto 1-z$ and $z\mapsto 1/z$. The corresponding modular maps yield curves in $\M_3$ and $\A_3$.

As for the \Yfamily{} $\Y$, for each $t\in\P^{*}$ the curves $\Y_{t}$ and $\Y_{1-t}$ are isomorphic via the map $(x,y)\mapsto(1-x,\zeta_{12}y)$, which induces the automorphism $z\mapsto 1-z$ on $\P^{1}$. Since any isomorphism between fibres $\Y_{t}$ and $\Y_{t'}$ must descend to an isomorphism of $\P^{1}$ interchanging branching values of the same order, it is clear that no other two fibres are isomorphic, and therefore the family is actually parametrised by $\P^{*}/\sim$, where $z\sim 1-z$. In~\autoref{subsec:PrymC6} we will give a more explicit description of this family in terms of its Prym-Torelli image.

The discussion above proves the following.

\begin{lemma}\label{lem:families}Let $\X$ and $\Y$ be the families defined above.
\begin{enumerate}
\item The map $\P^{*}\to\M_{3}$, $t\mapsto \X_{t}$ is of degree $6$. It ramifies over $\X_{-1}$ that has 3 preimages $\{\X_{t}:t=-1,1/2,2\}$ and $\X_{\zeta_{6}}$ that has 2 preimages $\{\X_{t}:t=\zeta_{6}^{\pm1}\}$.

    The only fibres with a cyclic group of automorphisms or order larger than $4$ are $\X_{-1}$ that admits a cyclic group of order 8 and $\X_{\zeta_{6}}$ that admits a cyclic group of order 12.

\item The map $\P^{*}\to\M_{3}$, $t\mapsto \Y_{t}$ is of degree $2$. It ramifies only over $\Y_{\nicefrac{1}{2}}$ that has a single preimage.

    The only fibre with a cyclic group of automorphisms of order larger than $6$ is $\Y_{\nicefrac{1}{2}}$ that admits a cyclic group of order 12.
\end{enumerate}
\end{lemma}

\begin{proof}[Proof of~\autoref{thm:families}] If $(X,[\omega])$ corresponds to an orbifold point on $\W$, then $X$ must belong to one of the families in~\autoref{tab:orders}.

First of all, note that curves of type (iii) admit an automorphism of order 4 with branching data $(0;4,4,4,4)$, and therefore they also belong to family (i). Similarly, those of type (iv) admit automorphisms of order 4 and 6 with branching data $(0;4,4,4,4)$ and $(0;2,3,3,6)$ respectively, and therefore they belong both to families (i) and (ii). As a consequence we can suppose that $X$ belongs either to (i) or (ii).

Let us suppose that $X$ is of type (i). Looking at the branching data, one can see that $X$ is necessarily isomorphic to one of the following two curves for some~$t\in\P^{*}$
    \begin{align*}
    y^4 &= x(x-1)(x-t)\,, \\
    y^4 &= x^{3}(x-1)^{3}(x-t)\,.
    \end{align*}

However curves of the second kind are always hyperelliptic, with hyperelliptic involution given by
\[\tau\colon (x,y)\mapsto\left(\frac{tx-t}{x-t},t(t-1)\frac{y}{(y-t)^2}\right).\]
 As points of $\W$ cannot correspond to hyperelliptic curves by~\autoref{nonhyp}, the curve $X$ is necessarily isomorphic to some~$\X_{t}$.

If $X$ is of type (ii), the branching data tells us that $X$ must be isomorphic to some fibre~$\Y_{t}$.

The claim about the order of the orbifold points follows from~\autoref{orbifoldscyclic} and \autoref{lem:families}.
\end{proof}

\begin{rem}
Let us note here that the special fibre $\X_{-1}$ is isomorphic to the Fermat curve $x^{4}+y^{4}+z^{4}=0$ and that the unique intersection point of the \Xfamily{} and the \Yfamily{}, that is $\X_{\zeta_{6}}\cong\Y_{\nicefrac{1}{2}}$, is isomorphic to the exceptional Wiman curve of genus 3 with affine equation $y^{3}=x^{4}+1$.
\end{rem}

\subsection{Differential forms}

By the considerations in~\autoref{sec:orbifoldbg}, we are only interested in differential forms with a single zero in a fixed point of the Prym involution.

\begin{lemma}\label{lem:forms}
Let $t\in\P^*$.
\begin{enumerate}
\item The space of holomorphic $1$-forms on each fibre $\X_{t}$ of the \Xfamily{} is generated by the $\alphaX^{*}$-eigenforms
    \[\omegaX_{1} = \frac{\d{x}}{y^{3}}\,,\qquad \omegaX_{2} = \frac{x\d{x}}{y^{3}}\,,\qquad \omegaX_{3} = \frac{\d{x}}{y^{2}}\,.\]
In particular, we obtain $\Omega(\X_{t})^{-}=\langle\omegaX_{1},\omegaX_{2}\rangle$ and $\Omega(\X_{t})^{+}=\langle\omegaX_{3}\rangle$ as $\rhoX$-eigenspaces.
\item The space of holomorphic $1$-forms on each fibre $\Y_{t}$ of the \Yfamily{} is generated by the $\alphaY^{*}$-eigenforms
    \[\omegaY_{1} = \frac{\d{x}}{y}\,,\qquad \omegaY_{2} = \frac{y\d{x}}{x(x-1)(x-t)}\,,\qquad \omegaY_{3} = \frac{y^{4}\d{x}}{x^{2}(x-1)^{2}(x-t)^{2}}\,.\]
In particular, we obtain $\Omega(\Y_{t})^{-}=\langle\omegaY_{1},\omegaY_{2}\rangle$ and $\Omega(\Y_{t})^{+}=\langle\omegaY_{3}\rangle$ as $\rhoY$-eigenspaces.
\end{enumerate}
\end{lemma}

\begin{proof}By writing their local expressions, one can check that all these forms are holomorphic. The action of $\rho$ can be checked in the affine coordinates.
\end{proof}

By analysing the zeroes one obtains the following lemma.

\begin{lemma}\label{lem:4foldforms}
Let $t\in\P^*$.
\begin{enumerate}
\item The forms in $\P\Omega(\X_{t})^{-}$  having a 4-fold zero at a fixed point of $\rhoX$ are
\begin{itemize}
    \item $\omegaX_{1}$ which has a zero at the preimage of~$\infty$,
    \item $\omegaX_{2}$ which has a zero at the preimage of~$0$,
    \item $-\omegaX_{1}+\omegaX_{2}$ which has a zero at the preimage of~$1$, and
    \item $-t\omegaX_{1}+\omegaX_{2}$ which has a zero at the preimage of~$t$.
\end{itemize}
They all form an orbit under~$\mathrm{Aut}(\X_{t})$.
\item For $t\neq 1/2$, the only form in $\P\Omega(\Y_{t})^{-}$ which has a 4-fold zero at a fixed point of $\rhoY$ is~$\omegaY_{2}$.
\end{enumerate}
\end{lemma}

\begin{proof}\textit{1.} For any $\X_t$, the preimages of $0$, $1$, $t$ and $\infty$ are the only fixed points of $\rhoX$. Using local charts, it is easy to see that these are the only forms with 4-fold zeroes at those points.

The last statement follows from the fact that $\Aut(\X_{t})$ permutes the preimages of $0$, $1$, $t$ and~$\infty$.

\textit{2.} Observe that the differential $\d{x}$ does not vanish on $\Y_{t}$ away from the preimages of $0$, $1$, $t$ and $\infty$. Under the parametrisations \autoref{cyclicparametrisation}, the local expression of $\d{x}$ around the preimages of 0 and 1 is $\d{x}=3z^{2}\d{z}$ and around the preimages of $t$ is $\d{x}=2z\d{z}$. Looking at the local expressions, one can see that $\omegaY_{1}$ has simple zeroes at the (four) preimages of $0$ and $1$, and $\omegaY_{2}$ has a 4-fold zero at infinity.

Again using local charts, it is easy to see that a form $u\omegaY_{1}+v\omegaY_{2}$, $u,v\in\C$, can have at most 2-fold zeroes at the preimages of~$t$.

On the other hand, if $u\omegaY_{1}+v\omegaY_{2}$ has a 4-fold zero at $\infty$, then the local expression above implies that~$u=0$.
\end{proof}

We can now state the converse of~\autoref{thm:families}.

\begin{prop}\label{thm:orbicrit}
Let $t\in\P^*$ and let $\O_D$ be some real quadratic order.
\begin{enumerate}
\item If $\PX$ admits proper real multiplication by $\O_D$ with $\omegaX_1$ as an eigenform then $\omegaX_2$, $-\omegaX_1+\omegaX_2$ and $-t\omegaX_1+\omegaX_2$ are also eigenforms and $(\X_t,\omegaX_1)$ corresponds to an orbifold point on~$\W$.

Moreover, if $\X_{t}\cong\X_{-1}$, then $(\X_t,\omegaX_1)$ is of order $4$; if $\X_{t}\cong\X_{\zeta_{6}}$, then $(\X_t,\omegaX_1)$ is of order $6$; otherwise, $(\X_t,\omegaX_1)$ is of order~$2$.
\item If $\PY$ admits proper real multiplication by $\O_D$ with $\omegaY_2$ as an eigenform then $(\Y_t,\omegaY_2)$ corresponds to an orbifold point on~$\W$.

Moreover, if $\Y_{t}=\Y_{\nicefrac{1}{2}}$, then $(\Y_t,\omegaY_2)$ is of order $6$; otherwise, $(\Y_t,\omegaY_2)$ is of order~$3$.
\end{enumerate}
\end{prop}

\begin{proof}By the previous lemma, if one of the four forms on $\X_{t}$ is an eigenform for some choice of real multiplication $\O_D\hookrightarrow\End\PX$, then the other three are also eigenforms for the choice of real multiplication conjugate by the corresponding automorphism. The statements about the points of higher order follow from \autoref{fermateigen} and \autoref{lem:actionbeta}.

The rest of the claims follows from~\autoref{thm:orbifoldpts} and~\autoref{lem:families}.
\end{proof}

\begin{rem}\label{rem:wollmilchsau}
Note that, while the \Xfamily{} $\X$ is the same curve inside $\M_3$ that is studied in \cite{HS} and \cite{moellerst}, the flat structures we consider on the fibres are different and the families are actually disjoint in $\Omega\M_3$. More precisely, we are interested in Prym-Teichmüller curves, i.e. a differential in the $-1$ eigenspace for the Prym involution, while the Wollmilchsau Teichmüller curve is constructed as a cover of the elliptic curve $\X_t/\rho$, i.e. has the flat structure of the differential in the $+1$ eigenspace. In particular, for our choices of differential $(\X_t,\omega_t)$, the (projection of the) $\SL_2(\R)$ orbit will never be the curve $\X$, but the Prym-Teichmüller curve $W_D(4)$ whenever the real-multiplication condition is satisfied.
\end{rem}

\subsection{Homology}

To calculate the Jacobians of the fibres of the \Xfamily{} $\X$ and the \Yfamily{} $\Y$, we also need a good understanding of their homology.

Consider again the general family $\ZZ\rightarrow\P^*$ introduced at the beginning of~\autoref{sec:cycliccovers}. Set $\P^*_t\coloneqq\P^*-\{t\}$ and $\ZZ^*_t\coloneqq\pi^{-1}(\P^*_t)$, where $\pi\colon\ZZ_t\rightarrow\P^1$ is the projection onto the $x$ coordinate. We thus obtain an unramified cover and the sequence
\[1\rightarrow\pi_1(\ZZ_t^*)\rightarrow\pi_1(\P^*_t)\rightarrow C_d\rightarrow 1,\]
where $C_d$ denotes the cyclic group of order $d$, is exact. Let $\sigma_P$ denote a simple counter-clockwise loop around the point $P\in\mathbb{P}^1$. Then $\pi_1(\P^*_t)$ is generated by $\sigma_0,\sigma_1,\sigma_t$ and $\sigma_\infty$ and their product is trivial. Observe that these four loops are mapped to elements of order $\nicefrac{d}{\gcd(a_{1},d)}$, $\nicefrac{d}{\gcd(a_{2},d)}$, $\nicefrac{d}{\gcd(a_{3},d)}$ and $\nicefrac{d}{\gcd(a_{4},d)}$ in~$C_d$ respectively.
Moreover, cycles in $\pi_1(\P^*_t)$ whose image in $C_d$ is trivial lift to cycles in~$\H_1(\ZZ_t,\Z)$.

For cycles $F,G\in\H_1(\ZZ_t,\Z)$, we pick representatives intersecting at most transversely and define the intersection number $F\cdot G\coloneqq\sum F_p\cdot G_p$, where the sum is taken over all $p\in F\cap G$ and for any such $p$, we define $F_p\cdot G_p\coloneqq+1$ if $G$ approaches $F$ \enquote{from the right in the direction of travel} and $F_p\cdot G_p\coloneqq-1$ otherwise, cf.~\autoref{CyclesC4}.

In the following, we identify $\Gal(\ZZ_t/\P^1)=C_d$ with the $d$th complex roots of unity and choose the generator $\alpha$ as $\exp(2\pi\i/d)$. Since all the fibres are topologically equivalent, let us suppose for simplicity $t\in\R$, $t>1$. Then, the simply-connected set $\P^1-[0,\infty]$ contains no ramification points and therefore has $d$ disjoint preimages $S_1,\dotsc,S_d$, which we call \emph{sheets} of $\ZZ_t$. These are permuted transitively by $\alpha$ and we choose the numbering so that $\alpha(S_{[n]})=S_{[n+1]}$, where $[n]\coloneqq n\bmod d$. The sheet changes are given by the monodromy: a path travelling around $0$ in a counter-clockwise direction on sheet $[n]$ continues onto sheet $[n+a_0]$ after crossing the interval $(0,1)$ and similarly for the other branch points.

We are now in a position to explicitly describe the fibrewise homology of $\X$ and~$\Y$.

Let $\FX$ denote the lift of $\sigma_1^{-1}\sigma_0$ that starts on sheet number $1$ of $\X_t$ and let $\GX$ denote the lift of $\sigma_t^{-1}\sigma_1$ that also starts on sheet $1$ (see~\autoref{CyclesC4}). Observe that~$\FX\cdot\GX=+1$.
\begin{figure}
\centering
\includegraphics[width=\textwidth]{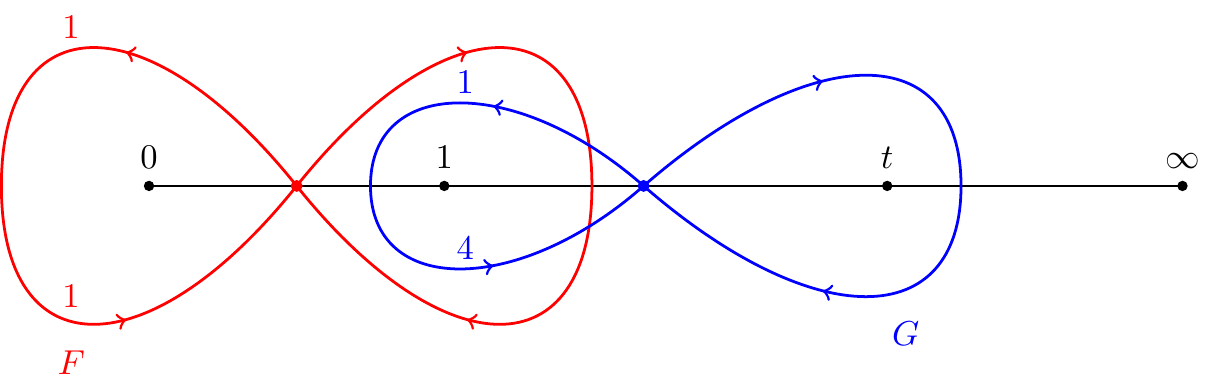}
\caption{The cycles $\FX$ and $\GX$ on $\X_{t}$. The upper-left parts of both cycles lie on sheet number 1. Observe that $\FX\cdot \GX=1$.}
\label{CyclesC4}
\end{figure}

Similarly, denote by $\FY$ and $\GY$ the lifts of $\sigma_1^{-1}\sigma_0$ and $\sigma_\infty^{-3}\sigma_t$, that start on sheet $1$ and $5$ of $\Y_t$, respectively (see~\autoref{CyclesC6}). Observe that $\FY\cdot\GY=0$.
\begin{figure}
\centering
\includegraphics[width=\textwidth]{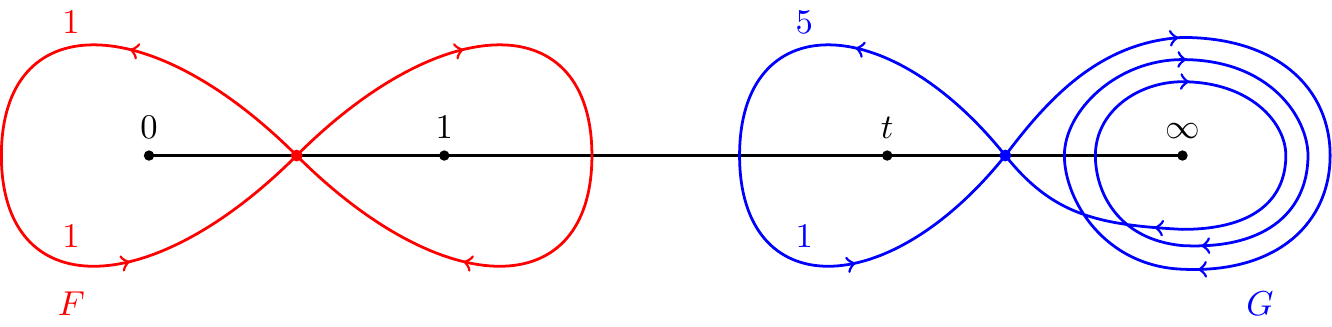}
\caption{The cycles $\FY$ and $\GY$ on $\Y_{t}$. The upper-left parts  lie on sheets number 1 and 5, respectively.}
\label{CyclesC6}
\end{figure}

\begingroup
\def\alphaX{\alpha}\def\alphaY{\alpha}
\def\FX{F}\def\FY{F}\def\GX{G}\def\GY{G}
\def\rhoX{\rho}\def\rhoY{\rho}

To ease notation, we will drop superscripts in the following lemma, as no confusion can arise.

\begin{lemma}\label{lem:hombasis} Let $t\in\P^*$.
\begin{enumerate}
\item The cycles $\FX,\alphaX \FX,\alphaX^2 \FX,\GX,\alphaX \GX,\alphaX^2 \GX$ yield a basis of $\H_1(\X_t,\Z)$. Moreover, the cycles
    \[\FX+\alphaX \FX+\GX+\alphaX \GX,\,
    -\GX+\alphaX^{2}\GX,\,
    \alphaX \FX+\alphaX^{2}\FX-\GX+\alphaX^{2}\GX,\,
    \FX+2\alphaX \FX+\alphaX^{2}\FX\]
span a $(1,2)$-polarised, $\rhoX$-anti-invariant sublattice of $\H_1(\X_t,\Z)$, which we denote by $\H_1^-(\X_t,\Z)$. The complementary $\rhoX$-invariant sublattice, $\H_1^+(\X_t,\Z)$, is spanned by $\FX+\alphaX^2 \FX, \GX+\alphaX^2 \GX$.
\item The cycles $F,\alphaY F,\alphaY^{3}F, \alphaY^{4}F,G,\alphaY G$ yield a basis of $\H_{1}(\Y_{t},\Z)$. Moreover, the cycles
    \[F-\alphaY^{3}F,\,
    \alphaY^{4}F-\alphaY F,\,
    G,\,
    \alphaY G\]
span a $(1,2)$-polarised, $\rhoY$-anti-invariant sublattice of $\H_{1}(\Y_{t},\Z)$, which we denote by $\H_1^-(\Y_t,\Z)$. The complementary $\rhoY$-invariant sublattice, $\H_1^+(\Y_t,\Z)$, is spanned by $F+\alphaY^{3}F, \alphaY^{4}F+\alphaY F$.
\end{enumerate}
\end{lemma}

\begin{proof}
\textit{1.} An elementary but somewhat tedious calculation yields the intersection matrix
\[\begin{pmatrix}
0&1&0&1&-1&0\\
-1&0&1&0&1&-1\\
0&-1&0&0&0&1\\
-1&0&0&0&1&0\\
1&-1&0&-1&0&1\\
0&1&-1&0&-1&0\\
\end{pmatrix}\]
for the above cycles on $\X_t$. As it has rank $6$ and determinant $1$, these cycles span all of $\H_1(\X_t,\Z)$. Furthermore, this immediately provides us with the relations
\begin{equation*}
\alphaX^3 \FX=-\FX-\alphaX \FX -\alphaX^2 \FX\quad\text{and}\quad\alphaX^3 \GX=-\GX-\alphaX \GX -\alphaX^2 \GX,
\end{equation*}
which confirms the claimed anti-invariance. The change to the second set of cycles yields
\[\begin{pmatrix}
0&0&1&0&&\\
0&0&0&2&&\\
-1&0&0&0&&\\
0&-2&0&0&&\\
&&&&0&2\\
&&&&-2&0\\
\end{pmatrix}\]
where the upper-left block is the anti-invariant and the lower-right block is the invariant part. Calculating determinants, we see that both blocks have determinant 4, proving the claim about the polarisation.

\textit{2.} Proceeding as before, one finds the following intersection matrix for the cycles on $\Y_{t}$
    \[\begin{pmatrix}
    0 & 0 & 0 & 1 & 0 & 0 \\
    0 & 0 & -1 & 0 & 0 & 0 \\
    0 & 1 & 0 & 0 & 0 & 0 \\
    -1 & 0 & 0 & 0 & 0 & 0 \\
    0 & 0 & 0 & 0 & 0 & 1 \\
    0 & 0 & 0 & 0 & -1 & 0 \\
    \end{pmatrix},\]
proving that they generate $\H_1(\Y_t,\Z)$, and the following one for the second set of cycles
\[\begin{pmatrix}
0&2&0&0&&\\
-2&0&0&0&&\\
0&0&0&1&&\\
0&0&-1&0&&\\
&&&&0&2\\
&&&&-2&0\\
\end{pmatrix},\]
yielding the $(1,2)\times(2)$-polarisation on the product.
\end{proof}
\endgroup

\subsection{Special points}
\label{specialpoints}

We briefly summarise some of the subtleties occurring at those points admitting additional symmetries.

\addsubsection{The curve \texorpdfstring{$\X_{2}$}{X\_2}}

In the \Xfamily{} $\X$, the fibres over $\nicefrac{1}{2}$, $-1$ and $2$ form an orbit under the action of $\Sym_3$. Over these points, $\alphaX$ extends to an automorphism $\betaX$ satisfying $\betaX^2=\alphaX$, i.e. a symmetry of order $8$, making them all isomorphic to the well-known \emph{Fermat curve}. More precisely, $\betaX$ may be obtained by lifting the automorphism that permutes two of the branch points and fixes the remaining pair on $\P^1$. Note that this may be achieved in two ways, e.g. for $t=2$, we obtain
\begin{align*}
\betaX_1\colon(x,y)&\mapsto\left(\frac{x}{x-1},\zeta_8\frac{y}{x-1}\right)\text{ and}\\
\betaX_2\colon(x,y)&\mapsto(2-x,\zeta_8y).
\end{align*}
Observe that $\betaX_1$ fixes $2$ and $0$ while interchanging $1$ and $\infty$, while $\betaX_2$ fixes $1$ and $\infty$ while interchanging $2$ and $0$. It is straight-forward to check the analogous statement of \autoref{lem:4foldforms} in this case.

\begin{lemma}\label{fermateigen}
Let $t$ be one of $\nicefrac{1}{2}$, $-1$ or $2$. Then the two forms from \autoref{lem:4foldforms} with zeros at the fixed points of $\betaX_1$ are eigenforms for $\betaX_1$, while the other two forms are eigenforms for $\betaX_2$.
\end{lemma}

\addsubsection{The curve \texorpdfstring{$\Y_{\nicefrac{1}{2}}$}{Y\_1/2} (or \texorpdfstring{$\X_{\zeta_{6}}$}{X\_zeta6})}

The only member of the \Yfamily{} $\Y$ whose automorphism group contains a cyclic group of order larger than 6 is $\Y_{\nicefrac{1}{2}}$, admitting an automorphism of order $12$, $\betaY(x,y)=(1-x,\zeta_{12}^{7}y)$, satisfying $\betaY^{2}=\alphaY$. In contrast to the case of $\X_2$, however, the automorphism $\betaY$ generates the full automorphism group.

Recall that, by~\autoref{thm:families}, the curve $\Y_{\nicefrac{1}{2}}$ is isomorphic to the curve $\X_{\zeta_{6}}$ of the \Xfamily. However, here we will use the model of the curve as a member of the \Yfamily.

Note first that $\betaY$ descends to the automorphism $z\mapsto 1-z$ of $\P^1$. Moreover $\betaY$ fixes $\infty$ with rotation number $\zeta_{12}$ and therefore $\betaY$ acts as $(1^{+},1^{-},2^{+},\dotsc,6^{+},6^{-})$ on the half-sheets, where we write $k^{+}$ (respectively $k^{-}$) for the upper half-plane (respectively lower half-plane) corresponding to the $k$th sheet.

By letting the initial points of $\FY$ and $\GY$ go to $1$ and $\infty$, respectively, and shrinking the cycles around the preimages of 0, 1, $t$ and $\infty$ one can use the  (equivalent) choice of cycles pictured in \autoref{CyclesShrunkC6}.

\begin{figure}[!htp]
\centering
\includegraphics{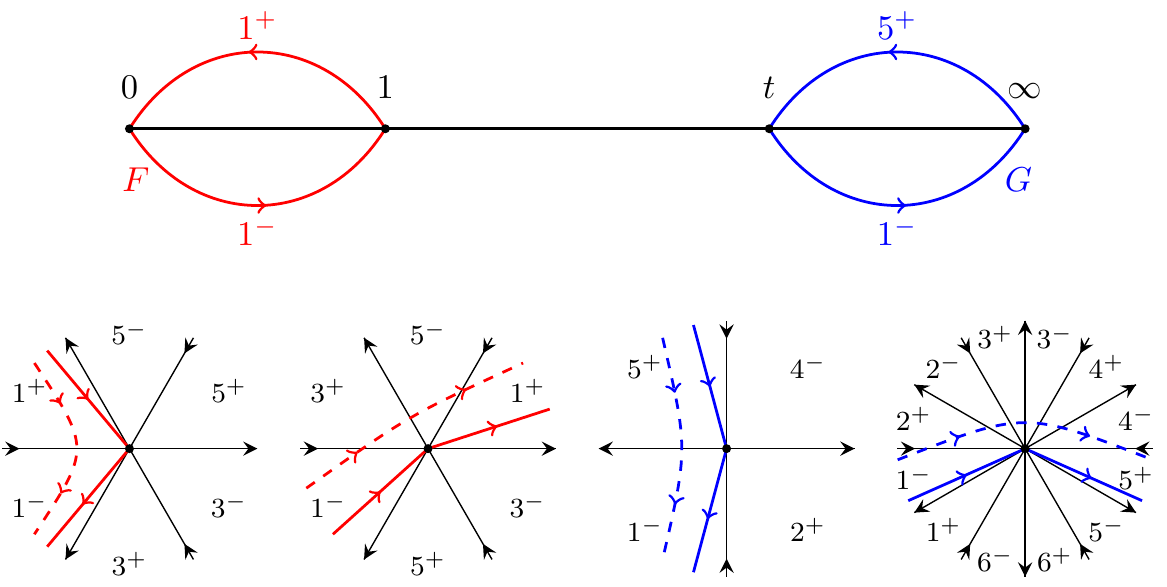}
\caption{The shrunk cycles $\FY$ and $\GY$, and the process of shrinking around the preimages of 0, 1, $t$ and $\infty$, respectively.}
\label{CyclesShrunkC6}
\end{figure}

After the shrinking process, the cycles $\FY$ and $\GY$ in~$\Y_{\nicefrac{1}{2}}$ have the shape depicted in~\autoref{cyclesC12}.

\begin{figure}[!htp]
\centering
\includegraphics{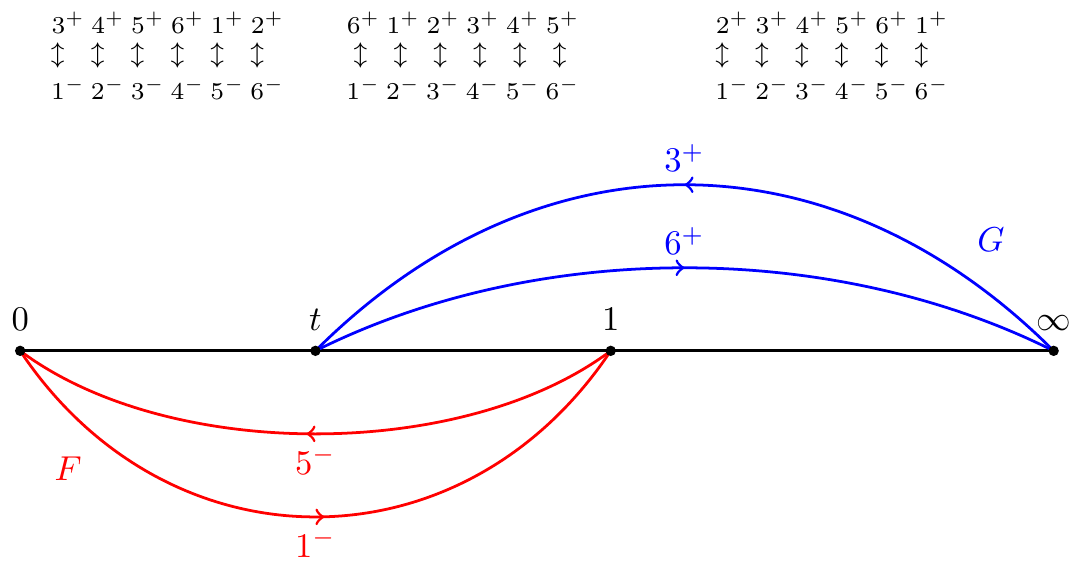}
\caption{The cycles $\FY$ and $\GY$ in $\Y_{\nicefrac{1}{2}}$.}
\label{cyclesC12}
\end{figure}

Taking all this into account, one can easily calculate the analytic and rational representations of $\betaY$.

\begin{lemma}\label{lem:actionbeta}The analytic and rational representations of $\betaY$ with respect to the bases $\H_{1}(\Y_{\nicefrac{1}{2}},\Z)=\langle \FY, \alphaY \FY, \alphaY^{3} \FY, \alphaY^{4} \FY, \GY, \alphaY \GY\rangle_{\Z}$ and $\Omega(\Y_{\nicefrac{1}{2}})=\langle\omegaY_{1},\omegaY_{2},\omegaY_{3}\rangle$ are given, respectively, by
    \[A_{\betaY}=\begin{pmatrix}
    \zeta_{12}^{-1} & 0 & 0 \\
    0 & \zeta_{12}^{7} & 0 \\
    0 & 0 & \zeta_{12}^{-2}
    \end{pmatrix}
    \qquad
    R_{\betaY}=\begin{pmatrix}
    0 & 0 & 1 & -1 & -1 & 0 \\
    0 & -1 & 1 & 1 & 0 & -1 \\
    1 & -1 & 0 & 0 & 1 & 0 \\
    1 & 1 & 0 & -1 & 0 & 1 \\
    0 & -1 & 0 & 1 & 1 & -1 \\
    1 & 1 & -1 & -1 & 1 & 2
    \end{pmatrix}.\]
In particular, $\omegaY_2$ is an eigenform for $\betaY$.
\end{lemma}

\subsection{Stable reduction of degenerate fibres}

While the \Yfamily{} is not compact in $\M_3$, it turns out that all fibres of its closure in $\overline{\M_3}$, the Deligne-Mumford compactification, admit compact Jacobians, i.e. that the Torelli image of $\overline{\Y}$ is contained in $\A_3$. Moreover, this analysis will be invaluable when constructing a fundamental domain for $\Y$ later.

\addsubsection{The degenerate fibres of \texorpdfstring{$\Y$}{Y}}

The degenerate fibres of the \Yfamily{} $\Y$ correspond to $t=0,1,\infty$. To describe them, we resort to the theory of \emph{admissible covers}. For a brief overview of the tools needed in this special case, see e.g. \cite[\S 4.1]{bouwmoeller} and the references therein.

The stable reduction when $t\to 1$ (equivalently, when $t\to 0$) yields the two components
    \begin{align*}
     \overline{\Y}_{1}^{1} &:y^{6}=x^{2}(x-1)^{5}\,,\ \mbox{ of genus 2,}\\
     \overline{\Y}_{1}^{2} &:y^{6}=x^{2}(x-1)^{3}\,,\ \mbox{ of genus 1.}
    \end{align*}

The stable reduction when $t\to\infty$ yields the three components
    \begin{align*}
     \overline{\Y}_{\infty}^{1} &:y^{6}=x^{2}(x-1)^{2}\,,\ \mbox{ consisting of two components of genus 1,}\\
     \overline{\Y}_{\infty}^{2} &:y^{6}=x^{3}(x-1)^{5}\,,\ \mbox{ of genus 1.}
    \end{align*}

A simple calculation gives the following lemma.

\begin{lemma}\label{lem:stablediff}The degeneration of the $\alphaY^{*}$-eigenforms of~\autoref{lem:forms} for $t\to 1$ is given by
    \[
     \omega_{1}^{1} = \frac{\d{x}}{y} \mbox{ on $\overline{\Y}_{1}^{1}$}\,,\qquad
     \omega_{2}^{1} = \frac{y\d{x}}{x(x-1)} \mbox{ on $\overline{\Y}_{1}^{2}$}\,,\qquad
     \omega_{3}^{1} = \frac{y^{4}\d{x}}{x^{2}(x-1)^{4}} \mbox{ on $\overline{\Y}_{1}^{1}$}\,,
    \]
and for $t\to\infty$ by
    \[
     \omega_{1}^{\infty} = \frac{\d{x}}{y} \mbox{ on $\overline{\Y}_{\infty}^{2}$}\,,\qquad
     \omega_{2}^{\infty} = \frac{y\d{x}}{x(x-1)} \mbox{ on $\overline{\Y}_{\infty}^{1}$}\,,\qquad
     \omega_{3}^{\infty} = \frac{\d{x}}{y^{2}} \mbox{ on $\overline{\Y}_{\infty}^{1}$}\,, \\
    \]
where the differentials are identically zero on the components where they are not defined.
\end{lemma}

\begin{figure}[!htp]
\centering
\includegraphics[width=\textwidth]{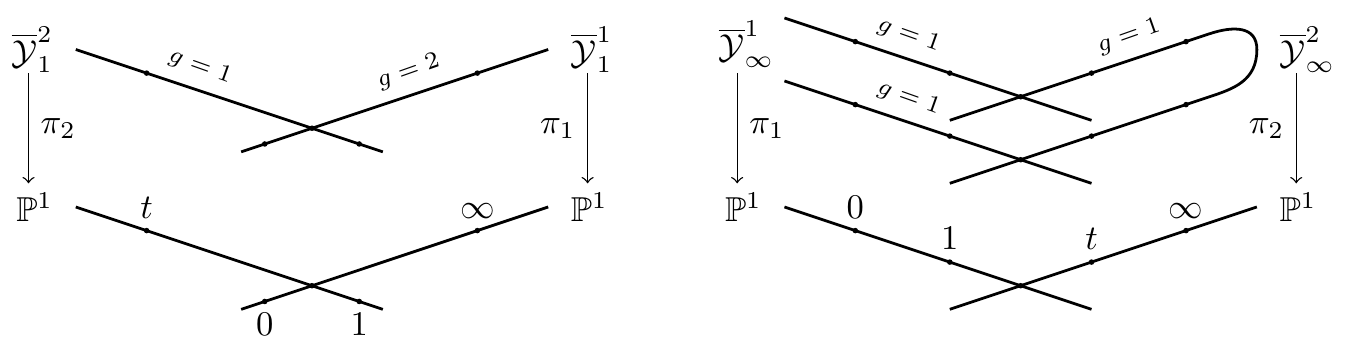}
\caption{The stable fibres $\overline{\Y}_{1}$ and $\overline{\Y}_{\infty}$.}
\label{StableReduction}
\end{figure}

Via the shrinking process introduced above, one can compute the degeneration of the cycles in both cases (see~\autoref{StableReductionCycles}). In the following lemma, we sum up some results about the homology of the degenerate fibres that we will need in \autoref{sec:Prym}.

\begin{lemma}\label{lem:stablecycles}Let $\dF{\infty},\dG{\infty}$ and $\dF{1},\dG{1}$ denote the cycles on $\overline{\Y}_{\infty}$ and $\overline{\Y}_{1}$ corresponding to the degeneration of $\FY$ and $\GY$.
\begin{enumerate}
    \item $\dF{\infty}$ and $\dG{\infty}$ live in $\overline{\Y}_{\infty}^{1}$ and $\overline{\Y}_{\infty}^{2}$ respectively.
    \item There is a decomposition of cycles $\dF{1}=\dF{1}_{1}+\dF{1}_{2}$ and $\dG{1}=\dG{1}_{1}+\dG{1}_{2}$, where $\dF{1}_{k},\dG{1}_{k}$ are cycles in the component $\overline{\Y}_{1}^{k}$ going through the nodal point. 

        Moreover, one has the following intersection matrices for the sets of cycles $\{\dF{1}_{k},\alphaY \dF{1}_{k},\alphaY^{3}\dF{1}_{k}, \alphaY^{4}\dF{1}_{k},\dG{1}_{k},\alphaY \dG{1}_{k}\}$, for $k=1,2$:
        \[ \begin{pmatrix}
            0&1&0&0&1&0\\
            -1&0&0&0&1&1\\
            0&0&0&1&-1&0\\
            0&0&-1&0&-1&-1\\
            -1&-1&1&1&0&2\\
            0&-1&0&1&-2&0\\
           \end{pmatrix}\mbox{ and }
           \begin{pmatrix}
            0&-1&0&1&-1&0\\
            1&0&-1&0&-1&-1\\
            0&1&0&-1&1&0\\
            -1&0&1&0&1&1\\
            1&1&-1&-1&0&-1\\
            0&1&0&-1&1&0\\
           \end{pmatrix}\,,\]
        respectively.
\end{enumerate}
\end{lemma}

\begin{figure}[!htb]
\centering
\includegraphics{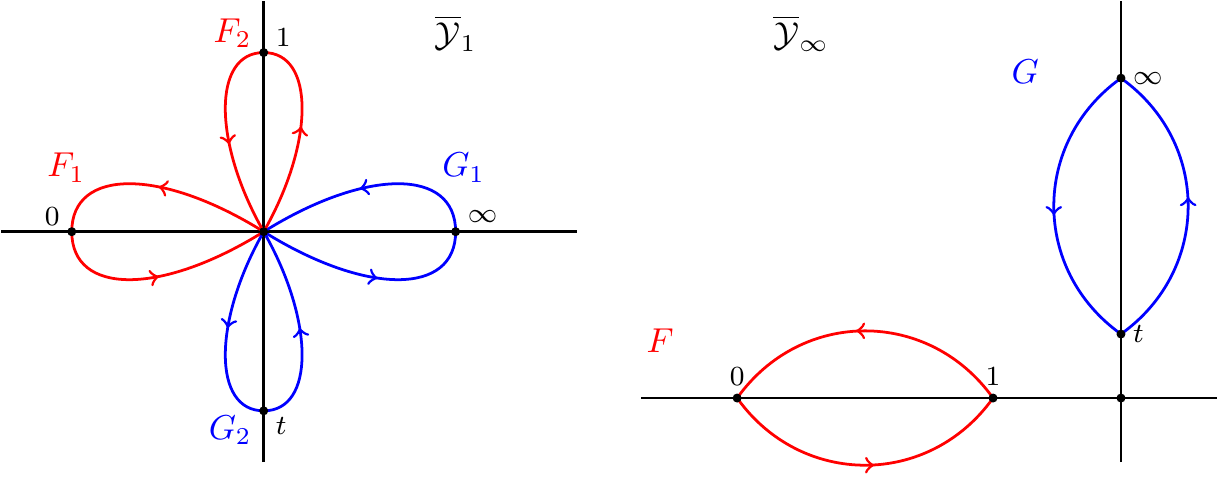}
\caption{The bases of homology in $\overline{\Y}_{1}$ and $\overline{\Y}_{\infty}$ as lifts of cycles in $\P^{1}$ by $\pi_{1}$ and $\pi_{2}$.}
\label{StableReductionCycles}
\end{figure}

\begin{proof}In the case of $\overline{\Y}_{\infty}$, it is obvious from~\autoref{CyclesC6} that the degeneration of the cycles $\FY$ and $\GY$ lie in $\overline{\Y}_{\infty}^{1}$ and $\overline{\Y}_{\infty}^{2}$ respectively.

The case $\overline{\Y}_{1}$ is more delicate. It follows again from~\autoref{CyclesC6} that the degeneration of both $\FY$ and $\GY$ are the union of cycles in $\overline{\Y}_{1}^{1}$ and $\overline{\Y}_{1}^{2}$ meeting at the nodal point. In fact, since the points in $\overline{\Y}_{1}$ corresponding to the preimages of 0 and 1 (respectively $t$ and $\infty$) lie in different components, it is clear that $\dF{1}$ (respectively $\dG{1}$) will decompose as the sum $\dF{1}_{1}+\dF{1}_{2}$ (respectively $\dG{1}_{1}+\dG{1}_{2}$) of cycles in $\overline{\Y}_{1}^{1}$ and $\overline{\Y}_{1}^{2}$.

Consider first the component $\overline{\Y}_{1}^{2}$, isomorphic to $y^{6}=x^{2}(x-1)^{3}$. Note that the preimages of $0$ and $1$ under $\pi_{2}$ correspond to the preimages of 1 and $t$ in the general member of our family $\Y_{t}$. Let us denote by $Q\in\overline{\Y}_{1}^{2}$ the nodal point and suppose, for simplicity, that its image $q\in\P^{1}$ under $\pi_{2}$ lies in the interval $[1,0]$. Removing this interval and proceeding as before we get the picture in~\autoref{degeneratecycles}, where the sheet changes follow from studying the behaviour of $\FY$ and $\GY$ around the preimages of 1 and $t$ in the general member of our family (see~\autoref{CyclesShrunkC6}).

\begin{figure}[!htp]
\centering
\includegraphics{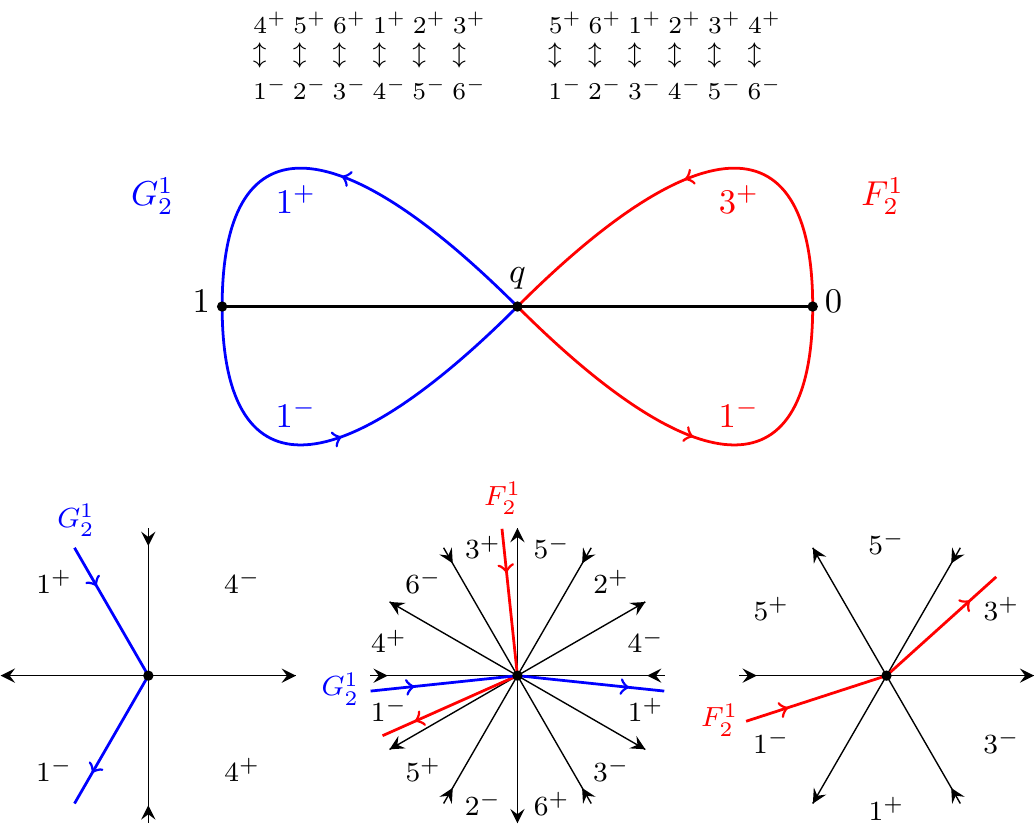}
\caption{Degenerate cycles on $\overline{\Y}_{1}^{2}$ and their behaviour around the preimages of $t$, $q$ and 1, respectively. Note that $\dF{1}_{2}\cdot \dG{1}_{2}=-1$.}
\label{degeneratecycles}
\end{figure}

One can get a similar picture for the other component $\overline{\Y}_{1}^{1}$. Now a tedious but straightforward calculation yields the intersection matrices.
\end{proof}


\section{The Prym-Torelli images}\label{sec:Prym}

To understand the orbifold points of $\W$, by~\autoref{thm:orbicrit}, we must determine which $\X_t$ and $\Y_t$ admit real multiplication that satisfies the eigenform condition. Therefore, the aim of this section is to concretely calculate the period matrices of the families of Prym varieties $\PX$ of the \Xfamily{} and $\PY$ of the \Yfamily.

\subsection{The Prym variety \texorpdfstring{$\PX$}{P(X\_t)}}\label{subsec:PrymC4}

\begingroup
\def\alphaX{\alpha}

In the case of the \Xfamily{} $\X$, all the fibres $\X_{t}$ are sent to the same Prym variety by the Prym-Torelli map.

\begin{prop}\label{prop:prymc4}
For all $t\in\P^*$,  the Prym variety $\PX$ is isomorphic to $\C^{2}/\Lambda$, where $\Lambda=\PrymX\cdot\Z^{4}$ for
    \[\PrymX= \begin{pmatrix}
    -\frac{1+\i}{2} & 1 & 1 & 0 \\
    1 & -(1+\i) & 0 & 2 \\
    \end{pmatrix}\,,\]
together with the polarisation induced by the intersection matrix
    \[\EX=\begin{pmatrix}
    0&0&1&0\\
    0&0&0&2\\
    -1&0&0&0\\
    0&-2&0&0\\
    \end{pmatrix}.\]

In particular, the image of the \Xfamily{} $\X$ under the Prym-Torelli map is a single point.
\end{prop}

Calculating Jacobians of curves with automorphisms can be done by a method attributed to Bolza, see~\cite[Chap. 11.7]{birkenhakelange} for details. The idea is to determine, for a given automorphism $\sigma$ and fixed choices of basis, the analytic and rational representations $A_\sigma$ and $R_\sigma$ of the automorphisms and use this information to find relations in the period matrix~$\Pi$, using the identity
\[\begin{pmatrix}A_\sigma&0\\0&\overline{A_\sigma}\end{pmatrix}\begin{pmatrix}\Pi\\\overline{\Pi}\end{pmatrix}=\begin{pmatrix}\Pi\\\overline{\Pi}\end{pmatrix}R_\sigma.\]

The group of automorphisms of a general member of the \Xfamily{} $\X$ is generated by $\alphaX\coloneqq\alpha^{\X}$ and the involutions
    \[\gamma\colon(x,y)\mapsto\left(\frac{t}{x},\frac{y\sqrt{t}}{x}\right)\quad\mbox{and}\quad
    \delta\colon(x,y)\mapsto\left(\frac{t(x-1)}{x-t},\frac{-y\sqrt{t(t-1)}}{x-t}\right)\,.\]

Note that $\gamma$ and $\delta$ are lifts by $\pi$ of the automorphisms of $\P^{1}$ given by $z\mapsto \frac{t}{z}$ and $z\mapsto \frac{tz-t}{z-t}$, respectively. In particular, these two involutions generate a Klein four-group acting on the fixed points of~$\rho^\X$.

By~\autoref{lem:forms}, the action of $\alphaX^{*}$ on $\Omega(\X_{t})$ is given by
    \[\begin{pmatrix} \i & 0 & 0 \\ 0 & \i & 0 \\ 0 & 0 & -1\\\end{pmatrix}\]
in the eigenform basis. The automorphisms $\gamma$ and $\delta$ induce analytic representations
    \[\gamma^{*}=\begin{pmatrix} 0 & -\sqrt{t} & 0 \\ -\dfrac{1}{\sqrt{t}} & 0 & 0 \\ 0 & 0 & -1 \end{pmatrix}\,\quad\text{and} \quad
    \delta^{*}=\begin{pmatrix} \dfrac{-t}{\sqrt{t(t-1)}} & \dfrac{-t}{\sqrt{t(t-1)}} & 0 \\ \dfrac{1}{\sqrt{t(t-1)}} & \dfrac{t}{\sqrt{t(t-1)}} & 0 \\ 0 & 0 & -1 \end{pmatrix}\,.
    \]

To calculate the rational representation, let us suppose again $t\in\R$, $t>1$. Keeping track of the action of $\gamma$ and $\delta$ on the branching points of $\pi$ and on the half-sheets of the cover, one can write down the action of these automorphisms in the homology
    \begin{align*}
     \gamma\FX&=-\alphaX^{2}\FX+\GX+\alphaX\GX\,,
     &\gamma\GX&=-\GX\,, \\
     \delta\FX&=-\FX\,,
     &\delta\GX&=-\alphaX\FX-\alphaX^{2}\FX-\alphaX^{2}\GX\,.
    \end{align*}

\begin{rem}\label{quotientEi}
Observe that $\gamma$ and $\delta$ act as involutions and the quotient is $\X_t/\gamma\cong\X_t/\delta\cong E_\i$, where $E_\i$ is the unique elliptic curve with an order four automorphism. Indeed, $\X_t$ is not hyperelliptic and $\delta$ and $\gamma$ have fixed points (e.g. preimages of $\sqrt{t}$ and $t-\sqrt{t(t-1)}$ on $\X_t$), therefore the quotient has genus $1$. Moreover, $\alpha$ commutes with both $\delta$ and $\gamma$, hence descends to an order four automorphism of the quotient elliptic curve.
\end{rem}

\begin{proof}[Proof of~\autoref{prop:prymc4}]
To calculate the Jacobian $\Jac(\X)$ write $f_{i}\coloneqq f^\X_{i}(t)=\int_{\FX}\omegaX_{i}$ and $g_{i}\coloneqq g_{i}(t)=\int_{\GX}\omegaX_{i}$. From the action of $\alphaX$ one can deduce that the Jacobian of $\X_{t}$ in the bases of Lemmas~\ref{lem:forms} and~\ref{lem:hombasis} is given by the period matrix
    \[\PiX_{t} = \begin{pmatrix}
    f_{1} & \i f_{1} & - f_{1} &  g_{1} & \i g_{1} & - g_{1} \\
    f_{2} & \i f_{2} & - f_{2} &  g_{2} & \i g_{2} & - g_{2} \\
    f_{3} & - f_{3} &  f_{3} &  g_{3} & - g_{3} &  g_{3} \\
    \end{pmatrix}.\]

Using the actions of $\gamma$ and $\delta$ both on $\Omega(\X_{t})$ and $\H_1(\X_t,\Z)$ one gets the relations
    \[f_{1} = -\sqrt{t} f_{2}- g_{1}(1+\i)\,,\quad
    g_{2} =  \frac{g_{1}}{\sqrt{t}}\,,\quad
    g_{1} = \frac{-f_{2}\sqrt{t}(1-\sqrt{t}+\sqrt{t-1})}{(1+\i)(\sqrt{t-1}-\sqrt{t})}\,.\]

By changing to the basis of $\H_1^-(\X_t,\Z) \oplus \H_1^+(\X_t,\Z)$ given in~\autoref{lem:hombasis} one gets
\[\begin{pmatrix}
(1+ \i) (f_{1}+g_{1}) & -2 g_{1} & -2 g_{1}+(\i-1) f_{1} & 2 \i f_{1} & 0 & 0 \\
(1+ \i) (f_{2}+g_{2}) & -2 g_{2} & -2 g_{2}+(\i-1) f_{2} & 2 \i f_{2} & 0 & 0 \\
0 & 0 & 0 & 0 & 2 f_{3} & 2 g_{3} \\
\end{pmatrix}\]
and sees that the Jacobian $\Jac(\X_t)$ is isogenous to the product $\PX\times\Jac(\X_t/\rhoX)$, where $\PX$ is $(1,2)$-polarised and $\Jac(\X_t/\rhoX)$ is $(2)$-polarised. Note that the polarisation on $\PX$ is given by the principal $4\times 4$ minor in the intersection matrix in the proof of~\autoref{lem:hombasis}, which agrees with $\EX$.

Finally, we can change the basis of $\Omega(\X_{t})^{-}$ by the matrix
    \begin{equation}\label{eq:Qt}
    Q_{t} = \dfrac{1}{\sqrt{t-1} f_{2}}\begin{pmatrix}
    -\dfrac{(1+\i)(\sqrt{t}-\sqrt{t-1})}{4\sqrt{t}} & -\dfrac{1+i}{4} \\
    \dfrac{\i}{2\sqrt{t}} & \dfrac{i(\sqrt{t}-\sqrt{t-1})}{2}
    \end{pmatrix},
    \end{equation}
to get the period matrix
    \[\begin{pmatrix}
    \PrymX & 0 \\
    0 & \EllX[t]
    \end{pmatrix}\,\mbox{ where }
    \PrymX:= \begin{pmatrix}
    -\frac{1+\i}{2} & 1 & 1 & 0 \\
    1 & -(1+\i) & 0 & 2 \\
    \end{pmatrix}\ \mbox{and}\
    \EllX[t]:=\begin{pmatrix}
    2  f_{3} &  2  g_{3} \\
    \end{pmatrix}.\]

Note that $\PrymX$ no longer depends on $t$, proving the final statement.
\end{proof}

\begin{rem}These results are equivalent to those of Gu\`{a}rdia in~\cite{guardia}. However, we cannot simply apply his results for two reasons. First, we are not restricted to real branching values and in particular the curve $\X_{\zeta_{6}}$ plays a special role. More importantly, in order to study the points of intersection with the Prym-Teichmüller curves $\W$, we need to keep track of the differential forms with a 4-fold zero in each fibre of the family. As a consequence, we need an explicit expression of the elements of $\Omega(\X_{t})^{-}(4)$, i.e. the $\rho$-anti-invariant differential forms with a 4-fold zero, in the basis in which the period matrix $\PrymX$ above is written.
\end{rem}

\addsubsection{The endomorphism ring \texorpdfstring{$\End\PX$}{End(P(X\_t))}}

To see when $\PX$ has real multiplication by a given order, we need a good understanding of the endomorphism ring. First, however, we describe the endomorphism algebra.

\begin{prop}\label{EndPX}
The endomorphism algebra $\End_{\Q}\PX$ is the algebra isomorphic to $M_{2}(\Q[\i])$ generated by the identity and the automorphisms $\alphaX$, $\gamma$, $\delta$ and $\gamma\delta$. 
\end{prop}

\begin{proof}Note that the automorphisms $\alphaX$, $\gamma$ and $\delta$ of $\X_{t}$ preserve the spaces $\Omega(\X_{t})^{-}$ and $\H_{1}^{-}(\X_{t},\Z)$, so they induce automorphisms of the Prym variety. One can construct their analytic and rational representations in the bases of Lemmas~\ref{lem:forms} and~\ref{lem:hombasis} to obtain
    \[\begin{array}{ll}
    A_{\alphaX}=\left(\begin{matrix}
    \i & 0 \\
    0 & \i
    \end{matrix}\right)\,,
    &
    R_{\alphaX}=\left(\begin{matrix}
    1 & -2 & -2 & 0 \\
    -1 & 1 & 0 & -2 \\
    2 & -2 & -1 & 2 \\
    -1 & 2 & 1 & -1
    \end{matrix}\right)\,; \\
    A_{\gamma}=\left(\begin{matrix}
    0 & \frac{1-\i}{2} \\
    1+\i & 0
    \end{matrix}\right)\,,\quad
    &
    R_{\gamma}=\left(\begin{matrix}
    1 & 0 & 0 & 2 \\
    1 & -1 & -1 & 0 \\
    0 & 0 & 1 & 2 \\
    0 & 0 & 0 & -1
    \end{matrix}\right)\,; \\
    A_{\delta}=\left(\begin{matrix}
    1 & 0 \\
    0 & -1
    \end{matrix}\right)\,,
    &
    R_{\delta}=\left(\begin{matrix}
    1 & 0 & 0 & 0 \\
    0 & -1 & 0 & 0 \\
    0 & 2 & 1 & 0 \\
    -1 & 0 & 0 & -1
    \end{matrix}\right)\,.
    \end{array}
    \]

Since $A_{\alphaX}$ lies in the centre of $M_{2}(\C)$ and the involutions $\gamma$ and $\delta$ anti-commute, the endomorphism algebra $\End_{\Q}\PX$ must contain the (definite) quaternion algebra $F=\langle A_{\alphaX}, A_{\gamma}, A_{\delta}\rangle_{\Q}\cong M_{2}(\Q[\i])$. It is easy to see that this already has to be the entire algebra $\End_{\Q}\PX$ (see~\cite[Prop. 13.4.1]{birkenhakelange}). In particular, any element of $\End_{\Q}\PX$ can be written as a $\Q[\i]$-linear combination of $\mathrm{Id}$, $A_{\gamma}$, $A_{\delta}$ and~$A_{\gamma\delta}$.
\end{proof}

Recall that, for any polarised abelian variety, the Rosati inovolution $\cdot'$ on the endomorphism ring is induced by the polarisation.
Therefore, given an element $\varphi\in\End_{\Q}\PX$ with rational representation $R_{\varphi}$, its image $\varphi'$ under the Rosati involution has rational representation $E^{-1}R_{\varphi}^{T}E$, where $E=E^\X$ is the polarisation matrix from above. It is then easy to check that $\alphaX'=-\alphaX$, $\gamma'=\gamma$, $\delta'=\delta$ and $(\gamma\delta)'=-\gamma\delta$. Under the embedding $F\hookrightarrow M_{2}(\C)$ given by the analytic representation, the Rosati involution is the restriction of the involution
    \begin{equation}\label{eq:C4involutionM2C}
    \begin{array}{ll}
    \begin{array}{lll}
    M_{2}(\C) & \to & M_{2}(\C) \\ B & \mapsto & A^{-1} B^{H}A
    \end{array} & ,\quad \mbox{for }A=\begin{pmatrix}2&0\\0&1\\\end{pmatrix}
    \end{array}
    \end{equation}
where $B^{H}$ denotes the hermitian transpose.

This gives us a simple criterion to check whether a specific rational endomorphism actually lies in $\End\PX$.

\endgroup


\subsection{The Prym variety \texorpdfstring{$\PY$}{P(Y\_t)}}\label{subsec:PrymC6}

\begingroup
\def\alphaY{\alpha}

In the case of the \Yfamily{} $\Y$, we have the following characterisation.

\begin{prop}\label{prop:prymc6}
For all $t\in\P^*$, the Prym variety $\PY=\C^{2}/\Lambda_{t}$, where $\Lambda_{t}=\PrymY[t]\cdot\Z^{4}$ for
    \[\PrymY[t]=\begin{pmatrix}
    2f  & 2\zeta_{6}^{2}f  & 1 & \zeta_{6}^{-1} \\
    2 & 2\zeta_{6}^{-2} & 2f  & 2\zeta_{6}f  \\
    \end{pmatrix}\,,\]
where $f\coloneqq f(t)=\int_{\FY}\omegaY_{1}$ is the period map, together with the polarisation induced by the intersection matrix
    \[\EY=\begin{pmatrix}
    0 & 2 & 0 & 0 \\
    -2 & 0 & 0 & 0 \\
    0 & 0 & 0 & 1 \\
    0 & 0 & -1 & 0 \\
    \end{pmatrix}.\]
\end{prop}

As above, we use Bolza's method for calculating the period matrix. Fortunately, in this case it suffices to regard~$\alphaY\coloneqq\alpha^{\Y}$.

By~\autoref{lem:forms}, the action of $\alphaY^{*}$ on $\Omega(\Y_{t})$ is given by
    \[\begin{pmatrix}
    \zeta_{6}^{-1} & 0 & 0 \\
    0 & \zeta_{6} & 0 \\
    0 & 0 & \zeta_{6}^{4}
    \end{pmatrix}\]
in the eigenform basis.

\begin{proof}[Proof of~\autoref{prop:prymc6}]
Again, we write $f_{i}\coloneqq f^{\Y}_{i}(t)=\int_{\FY}\omegaY_{i}$ and $g_{i}\coloneqq g^{\Y}_{i}(t)=\int_{\GY}\omegaY_{i}$. Since $\alphaY^{3}(\GY)=\rhoY(\GY)=-\GY$ and $\rhoY^{*}\omegaY_{3}=-\omegaY_{3}$, one has $g_{3}=0$. Using the action of $\alphaY$ on $\Omega(\Y_{t})$, one gets that, in these bases, the period matrix of $\Y_{t}$ reads
\begin{equation}\label{eq:pmYparam}
    \PiY_{t}=\begin{pmatrix}
    f_{1} & \zeta_{6}^{-1}f_{1} & -f_{1} & \zeta_{6}^{2}f_{1} & g_{1} & \zeta_{6}^{-1}g_{1} \\
    f_{2} & \zeta_{6}f_{2} & -f_{2} & \zeta_{6}^{-2}f_{2} & g_{2} & \zeta_{6}g_{2} \\
    f_{3} & \zeta_{6}^{-2}f_{3} & f_{3} & \zeta_{6}^{-2}f_{3} & 0 & 0 \\
    \end{pmatrix}.
\end{equation}

Moreover, by normalising $g_{1}=f_{2}=f_{3}=1$ and using Riemann's relations, one sees that
    \begin{align*}
    \PiY_{t}E^{-1}\bigl(\PiY_{t}\bigr)^{T}=0 & \Rightarrow g_{2}=2f_{1},\text{ and} \\
    \i\PiY_{t}E^{-1}\bigl(\overline{\PiY_{t}}\bigr)^{T}>0 & \Rightarrow  2|f_{1}|^{2}-1<0.
    \end{align*}

Writing $f:=f_{1}$, we finally get
    \begin{equation}\label{eq:periodmatrixY}
    \PiY_{t}=\begin{pmatrix}
    f & \zeta_{6}^{-1}f & -f & \zeta_{6}^{2}f & 1 & \zeta_{6}^{-1} \\
    1 & \zeta_{6} & -1 & \zeta_{6}^{-2} & 2f & 2\zeta_{6}f \\
    1 & \zeta_{6}^{-2} & 1 & \zeta_{6}^{-2} & 0 & 0 \\
    \end{pmatrix}.
    \end{equation}

As above, the Jacobian $\Jac(\Y_{t})$ is isogenous to the variety $\PY\times \Jac(\Y_{t}/\rhoY)$, whose period matrix is obtained by changing to the basis of $\H_{1}^{-}(\Y_{t},\Z)\oplus \H_{1}^{+}(\Y_{t},\Z)$ of \autoref{lem:hombasis}, yielding
    \[\begin{pmatrix}
    \PrymY[t] & 0 \\
    0 & \EllY[t]
    \end{pmatrix},\mbox{ where }
    \PrymY[t]\coloneqq \begin{pmatrix}
    2f  & 2\zeta_{6}^{2}f  & 1 & \zeta_{6}^{-1}  \\
    2 & 2\zeta_{6}^{-2} & 2f  & 2\zeta_{6}f
    \end{pmatrix}\mbox{ and }
    \EllY[t]\coloneqq\begin{pmatrix}
    2 & 2\zeta_{6}^{-2} \\
    \end{pmatrix}.\]

The polarisation on $\PY$ is again given by the principal $4\times 4$ minor in the intersection matrix in the proof of~\autoref{lem:hombasis}, which agrees with $\EY$.
\end{proof}

\addsubsection{The endomorphism ring \texorpdfstring{$\End\PY$}{End(P(Y\_t))}}\label{subsec:EndPY}

In this section we study the endomorphism ring $\End\PY$ and the endomorphism algebra $\End_{\Q}\PY$ in order to get a description of the \Yfamily{} $\Y$ as a Shimura curve. More precisely, let $\mathfrak{M}$ denote the maximal order
    \begin{equation}\label{eq:maximalorder}
    \mathfrak{M}=\Z\left[\frac{\mathbf{1}+\mathbf{j}}{2}, \frac{\mathbf{1}-\mathbf{j}}{2}, \frac{\mathbf{i}+\mathbf{ij}}{2}, \frac{\mathbf{i}-\mathbf{ij}}{2}\right]
    \end{equation}
in the quaternion algebra
    \[F\coloneqq\left\{x_{0}+x_{1}\mathbf{i}+x_{2}\mathbf{j}+x_{3}\mathbf{ij}\ :\ x_{k}\in\Q\,,\,\mathbf{i}^{2}=2\,,\,\mathbf{j}^{2}=-3\right\}
    \cong\left(\frac{2,-3}{\Q}\right).\]

We will prove the following.

\begin{prop}\label{prop:shimura}The Prym-Torelli map gives an isomorphism between the compactification $\overline{\Y}$ of the \Yfamily{} $\Y$ and the (compact) Shimura curve whose points correspond to abelian surfaces with a $(1,2)$ polarisation, endomorphism ring $\End A \cong\mathfrak{M}$ and Rosati involution given by~\autoref{eq:C6Rosati}. This curve is isomorphic to $\HH/\Delta(2,6,6)$.
\end{prop}

Recall that a (compact hyperbolic) triangle group is a Fuchsian group constructed in the following way. Let $l$, $m$ and $n$ be positive integers such that $1/l+1/m+1/n<1$ and consider a hyperbolic triangle $T$ in the hyperbolic plane with vertices $v_{l}$, $v_{m}$ and $v_{n}$ with angles $\pi/l$, $\pi/m$ and $\pi/n$ respectively. The subgroup $\Delta(l,m,n)$ of $\PSL_{2}(\R)$ generated by the positive rotations through angles $2\pi/l$, $2\pi/m$ and $2\pi/n$ around $v_{l}$, $v_{m}$ and $v_{n}$ respectively is called a \emph{triangle group of signature $(l,m,n)$}. The triangle $T$ is unique up to conjugation in $\PSL_{2}(\R)$ and, therefore, so is the associated triangle group described above (\cite[\S7.12]{beardon}). Note that the quadrilateral consisting of the union of the triangle $T$ and any of its reflections serves as a fundamental domain for $\Delta(l,m,n)$ (see~\autoref{fig:trianglegroup} for a fundamental domain of $\Delta(2,6,6)$ inside the hyperbolic disc~$\D$).

\medskip

Let us now calculate $\End\PY$. Since the automorphism $\alphaY$ of $\Y_{t}$ induces an automorphism of $\PY$ and $\mathbf{j}\coloneqq 2\alphaY-1$ satisfies $\mathbf{j}^{2}=-3$, there is always an embedding $\Q(\sqrt{-3})\hookrightarrow\End_{\Q}\PY$. However, the full endomorphism algebra of an abelian surface is never an imaginary quadratic field (see~\cite[Ex. 9.10(4)]{birkenhakelange}, for example) and one can check that the analytic and rational representations $A_{\mathbf{i}}$ and $R_{\mathbf{i}}$ defined below yield an element of $\End_{\Q}\PY$. It is then easy to see that the endomorphism algebra of the general member of our family agrees with the (indefinite) quaternion algebra $F$. 

Abelian varieties with given endomorphism structure have been intensely studied, notably by Shimura \cite{shimura}. Shimura explicitly constructs moduli spaces for such families in much greater generality than we require here. However, his results specialise to our situation.
To emulate his construction, we begin by observing that since $F\otimes\R\cong M_{2}(\R)$, we can see $F$ as a subalgebra of $M_{2}(\R)$. The following matrices show the relation between the embedding $F\hookrightarrow M_{2}(\R)$, the analytic representation $A\colon F\hookrightarrow M_{2}(\C)$ and the rational representation $R\colon F\hookrightarrow M_{4}(\Q)$
    \[\begin{array}{lll}
    \mathbf{i}=\begin{pmatrix}
    \sqrt{2} & 0 \\
    0 & -\sqrt{2}
    \end{pmatrix}\,,
    &
    A_{\mathbf{i}}=\begin{pmatrix}
    0 & 1 \\
    2 & 0
    \end{pmatrix}\,,
    &
    R_{\mathbf{i}}=\begin{pmatrix}
    0 & 0 & 1 & 1 \\
    0 & 0 & 0 & 1 \\
    2 & -2 & 0 & 0 \\
    0 & 2 & 0 & 0
    \end{pmatrix}; \\
    \mathbf{j}=\begin{pmatrix}
    0 & \sqrt{3} \\
    -\sqrt{3} & 0
    \end{pmatrix}\,,
    &
    A_{\mathbf{j}}=\begin{pmatrix}
    -i\sqrt{3} & 0 \\
    0 & i\sqrt{3}
    \end{pmatrix}\,,\quad
    &
    R_{\mathbf{j}}=\begin{pmatrix}
    -1 & 2 & 0 & 0 \\
    -2 & 1 & 0 & 0 \\
    0 & 0 & -1 & -2 \\
    0 & 0 & 2 & 1
    \end{pmatrix}; \\
    \mathbf{ij}=\begin{pmatrix}
    0 & \sqrt{6} \\
    \sqrt{6} & 0
    \end{pmatrix}\,,
    &
    A_{\mathbf{ij}}=\begin{pmatrix}
    0 & i\sqrt{3} \\
    -2i\sqrt{3} & 0
    \end{pmatrix}\,,
    &
    R_{\mathbf{ij}}=\begin{pmatrix}
    0 & 0 & 1 & -1 \\
    0 & 0 & 2 & 1 \\
    2 & 2 & 0 & 0 \\
    -4 & 2 & 0 & 0
    \end{pmatrix}.
    \end{array}
    \]

By checking which elements of $F$ have integral rational representation, one can see that the endomorphism ring $\End_{\Q}\PY$ of the general member of our family agrees with the maximal order $\mathfrak{M}$ defined above.

Proceeding as in the case of the \Xfamily{} and writing $x=x_{0}+x_{1}\mathbf{i}+x_{2}\mathbf{j}+x_{3}\mathbf{ij}$ for an element of $F$, we note that, by the Skolem-Noether theorem, the quaternion conjugation and the Rosati involution are conjugate. It is not difficult to check that, here, the Rosati involution is given by
    \begin{equation}\label{eq:C6Rosati}
    x':=\mathbf{j}^{-1}\overline{x}\mathbf{j} = x_{0}+x_{1}\mathbf{i}-x_{2}\mathbf{j}+x_{3}\mathbf{ij}\,,
    \end{equation}
where $\overline{x}=x_{0}-x_{1}\mathbf{i}-x_{2}\mathbf{j}-x_{3}\mathbf{ij}$ is the usual conjugation in $F$. Note that the Rosati involution in $F\hookrightarrow M_{2}(\R)$ agrees with transposition and that, under the embedding $F\hookrightarrow M_{2}(\C)$ given by the analytic representation, it is again the restriction of the involution
    \begin{equation}\label{eq:C6involutionM2C}
    \begin{array}{ll}
    \begin{array}{lll}
    M_{2}(\C) & \to & M_{2}(\C) \\ B & \mapsto & A^{-1} B^{H}A
    \end{array} & ,\quad \mbox{for }A=\begin{pmatrix}2&0\\0&1\\\end{pmatrix}\,.
    \end{array}
    \end{equation}

\begin{proof}[Proof of~\autoref{prop:shimura}]Let us give the construction of the \Yfamily{} as a Shimura curve. Following~\cite{shimura}, one can define the isomorphism
    \[\begin{array}{ccccl}
    \Phi: & \mathfrak{M} & \longrightarrow & \Lambda_{t} & \\
    & a & \longmapsto & A_{a}\cdot y\,, & \quad y=\left(\begin{smallmatrix}2f \\ 2 \end{smallmatrix}\right)
    \end{array}\]
where $A_{a}$ denotes the analytic representation of $a$, and check that the polarisation satisfies $E(\Phi(a),y)=\mathrm{tr}(a\cdot T)$ for $T=\frac{1}{3}\mathbf{j}\in F$. The family of abelian varieties $A$ with a $(1,2)$ polarisation together with an embedding $\mathfrak{M}\hookrightarrow\End A$ and Rosati involution induced by~\autoref{eq:C6Rosati} is then given by the Shimura curve $\HH/\Gamma(T,\mathfrak{M})$, where $\Gamma(T,\mathfrak{M})$ agrees with the group of elements of norm 1 of $\mathfrak{M}$. By~\cite{takeuchi} this is a quadrilateral group of signature~$\langle 0; 2,2,3,3\rangle$.

However, for each such variety $A$, there exist two different embeddings $F \hookrightarrow \mathrm{End}_{\Q} A$ which differ by quaternion conjugation on $F$. As a consequence, the map $\HH/\widetilde{\Gamma}(T,\mathfrak{M}) \to \A_{2,(1,2)}$ has degree 2, and the Shimura curve constructed above is a double cover of its image, which is uniformised by the triangle group $\Delta(2,6,6)$ extending~$\Gamma(T,\mathfrak{M})$ (see~\cite{takeuchi}).

Now, the Prym-Torelli image of $\overline{\Y}$ lies entirely in this family and the proposition follows.
\end{proof}

\begin{rem} Cyclic coverings of this type are well-known and have been intensely studied. For example, it immediately follows from the results of Deligne and Mostow \cite[\S 14.3]{delignemostow} that the compactified \Yfamily{} $\overline{\Y}$ is parametrised by~$\HH/\Delta(2,6,6)$.
 More precisely, the monodromy data of the \Yfamily{} yields (using their notation) $\mu_1=\mu_2=\nicefrac{1}{3}$, $\mu_3=\nicefrac{1}{2}$, and $\mu_4=\nicefrac{5}{6}$, hence we obtain a map from $\P^{1}$ into $\HH/\Delta(3,6,6)$. Taking the quotient by the additional symmetry in the branching data here present, it descends to a map from the basis of $\Y$ into $\HH/\Delta(2,6,6)$, as above.
\end{rem}

In our case, the lift of the period map $f=f(t)$ from $\P^{1}$ to the disc of radius $1/\sqrt{2}$ gives us a particular model of the Shimura curve introduced above as the quotient of this disc with the hyperbolic metric by the action of a specific triangle group $\Delta(2,6,6)$. In order to find a fundamental domain for this group, we will study the value of the period map at the special points of the compactification $\overline{\Y}$ of the \Yfamily{} $\Y$, namely the curves $\Y_{\nicefrac{1}{2}}$, $\overline{\Y}_{1}$ and $\overline{\Y}_{\infty}$. In particular, we will prove the following.

\begin{prop}\label{prop:trianglegroup}The $\Delta(2,6,6)$ group uniformising $\overline{\Y}$ is generated by the hyperbolic triangle with vertices $f(\nicefrac{1}{2})=\frac{3-\sqrt{3}+\i(\sqrt{3}-1)}{4}$ of angle $\nicefrac{\pi}{2}$, $f(1)=\frac{1}{2}\zeta_{6}$ of angle $\nicefrac{\pi}{6}$ and $f(\infty)=0$ of angle $\nicefrac{\pi}{6}$ inside the disc of radius $\nicefrac{1}{\sqrt{2}}$ (see~\autoref{fig:trianglegroup}). These vertices correspond to the curves $\Y_{\nicefrac{1}{2}}$, $\overline{\Y}_{1}$ and~$\overline{\Y}_{\infty}$ respectively.
\end{prop}

\begin{figure}[!htp]
\centering
\includegraphics{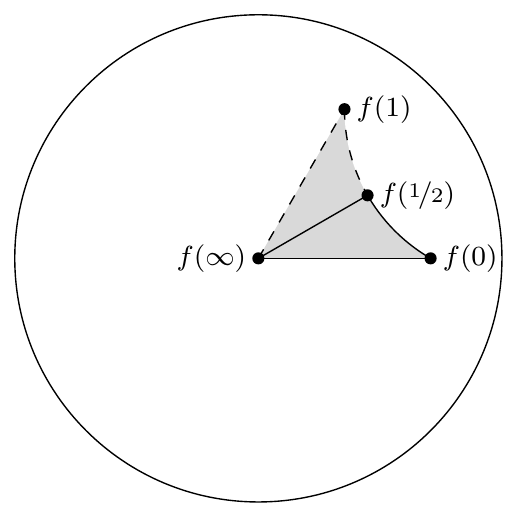}
\caption{Fundamental domain of $\Delta(2,6,6)$ on the disc of radius $\nicefrac{1}{\sqrt{2}}$ with vertices 0, $\nicefrac{1}{2}$ and $\frac{1}{4}(3-\sqrt{3})+\frac{\i}{4}(\sqrt{3}-1)$ corresponding to special fibres of $\overline{\Y}$.}
\label{fig:trianglegroup}
\end{figure}

\begin{proof}
It follows from~\autoref{lem:families}(2) that the curve $\Y_{\nicefrac{1}{2}}$ corresponds to the point of order 2 in the triangle group and, therefore, $\overline{\Y}_{1}$ and $\overline{\Y}_{\infty}$ correspond to the two points of order 6. Consider~\autoref{eq:periodmatrixY}, giving the period matrix $\PiY_{t}$ of the general member of the \Yfamily.

In the case of $\overline{\Y}_{\infty}$, it follows from~\autoref{lem:stablediff} and~\autoref{lem:stablecycles} that $\int_{\FY}\omega_{1}^{\infty} = \int_{\GY}\omega_{2}^{\infty} = \int_{\GY}\omega_{3}^{\infty} = 0$ and one has the following period matrix
    \[\PiY_{\infty}=\begin{pmatrix}
    0 & 0 & 0 & 0 & 1 & \zeta_{6}^{-1} \\
    1 & \zeta_{6} & -1 & \zeta_{6}^{-2} & 0 & 0 \\
    1 & \zeta_{6}^{-2} & 1 & \zeta_{6}^{-2} & 0 & 0 \\
    \end{pmatrix}.\]

In particular $f(\infty)=0$.

Similarly, using~\autoref{lem:actionbeta} and the fact that $A_{\betaY}\PiY_{\nicefrac{1}{2}}=\PiY_{\nicefrac{1}{2}}R_{\betaY}$ one gets
    \[\PiY_{\nicefrac{1}{2}}=\begin{pmatrix}
    \vartheta & \zeta_{6}^{-1}\vartheta & -\vartheta & \zeta_{6}^{2}\vartheta & 1 & \zeta_{6}^{-1} \\
    1 & \zeta_{6} & -1 & \zeta_{6}^{-2} & 2\vartheta & 2\zeta_{6}\vartheta \\
    1 & \zeta_{6}^{-2} & 1 & \zeta_{6}^{-2} & 0 & 0 \\
    \end{pmatrix}\ ,\]
where
    \[\vartheta = \frac{3-\sqrt{3}+\i(1-\sqrt{3})}{4}\,. \]

Finally, in the case $\overline{\Y}_{1}$, it follows again from~\autoref{lem:stablecycles} that $\dG{1}_{2} = \alphaY\dF{1}_{2}-\dF{1}_{2}$. Comparing this with the entries of the period matrix $\PiY_{t}$ in~\autoref{eq:pmYparam}, one finds that $2f(1)=g_{2}(1)=\zeta_{6}-1$. Therefore $f(1)=\frac{1}{2}\zeta_{6}^{2}$ and
    \[\PiY_{1}=\begin{pmatrix}
    \frac{1}{2}\zeta_{6}^{2} & \frac{1}{2}\zeta_{6} & -\frac{1}{2}\zeta_{6}^{2} & \frac{1}{2}\zeta_{6}^{-2} & 1 & \zeta_{6}^{-1} \\
    1 & \zeta_{6} & -1 & \zeta_{6}^{-2} & \zeta_{6}^{2} & -1 \\
    1 & \zeta_{6}^{-2} & 1 & \zeta_{6}^{-2} & 0 & 0 \\
    \end{pmatrix}.\]

Now, since $f(\infty)=0$ is a point of order 6 of $\Delta(2,6,6)$, the point $\frac{1}{2}\zeta_{6}^{2}$ corresponding to $\overline{\Y}_{1}$ (respectively the point $\vartheta$
corresponding to $\Y_{\nicefrac{1}{2}}$) is equivalent to $f(1)=\frac{1}{2}\zeta_{6}$ (respectively to $f(\nicefrac{1}{2})=\overline{\vartheta}$), by reflecting along the sides of the triangle. We may therefore choose our fundamental domain as claimed.
\end{proof}

\endgroup

\section{Orbifold points in \texorpdfstring{$\W$}{W\_D}}\label{sec:orbifoldpts}

In this section we will finally determine the orbifold points on $\W$. By~\autoref{thm:orbicrit}, these correspond precisely to the fibres of the \Xfamily{} $\X$ and of the \Yfamily{} $\Y$ whose Prym variety admits proper real multiplication by $\O_{D}$, together with an eigenform for real multiplication having a 4-fold zero at a fixed point of the Prym involution. Remember that $\O_{D}$ is defined as $\Z[T]/(T^{2}+bT+c)$, where $D=b^{2}-4c$. In particular, $\O_{D}$ is generated as a $\Z$-module by
    \[T\coloneqq\left\{\begin{array}{ll}
    \dfrac{\sqrt{D}}{2}\ , & \mbox{if $D\equiv 0\bmod{4}$}\,;\\
    \dfrac{1+\sqrt{D}}{2}\ , & \mbox{if $D\equiv 1\bmod{4}$\,.}
    \end{array}\right.\]

Let $D$ be a discriminant with conductor $f_{0}$ and let us recall the sets
\begin{align*}
\Hcal_{2}(D)\coloneqq \{(a,b,c) \in \Z^{3}\ :&\  a^{2}+b^{2}+c^{2}=D\ ,\ \gcd(a,b,c,f_{0})=1\,\},\text{ and} \\[0.2cm]
\Hcal_{3}(D)\coloneqq \{(a,b,c)\in\Z^{3}\ :&\
      2a^{2}-3b^{2}-c^{2}=2D\ ,\ \gcd(a,b,c,f_{0})=1\ ,\\
      &\ -3\sqrt{D}<a<-\sqrt{D}\ ,\ c < b \le 0\ , \\ 
      &\ (4a - 3b - 3c <0)\vee(4a - 3b - 3c=0\ \wedge\ c < 3b)\,\}
\end{align*}
defined in~\autoref{sec:intro}.

The number of orbifold points on $\W$ of orders 2, 3, 4 and 6 are given by the following formulas.

      \[\begin{array}{lll} e_{2}(D)\coloneqq & \left\{\begin{array}{l}
      0\ , \\
      |\Hcal_{2}(D)|/24\ ,
      \end{array}\right. &
      \begin{array}{l}
      D\equiv1\bmod{4}\quad\mbox{or}\quad D=8,12\,;\\
      \mbox{otherwise}\,;
      \end{array}
      \\[0.6cm]
      e_{3}(D)\coloneqq & \left\{\begin{array}{l}
      0\ , \\
      |\Hcal_{3}(D)|\ ,
      \end{array}\right. &
      \begin{array}{l}
      D=12\,;\\
      \mbox{otherwise}\,;
      \end{array}
      \\[0.6cm]
      e_{4}(D)\coloneqq & \left\{\begin{array}{l}
      1\ , \\
      0\ ,
      \end{array}\right. &
      \begin{array}{l}
      D=8\,; \\
      \mbox{otherwise}\,;
      \end{array}
      \\[0.6cm]
      e_{6}(D)\coloneqq & \left\{\begin{array}{l}
      1\ , \\
      0\ ,
      \end{array}\right. &
      \begin{array}{l}
      D=12\,; \\
      \mbox{otherwise}\,.
      \end{array}\\
      \end{array}\]


\subsection{Points of order 2 and 4}

\begin{theorem}\label{thm:orbpt2and4}The curve $\W[D=8]$ has one orbifold point of order 4. Moreover, no other $\W$ has orbifold points of order 4.

Let $D\neq 8,12$ be a discriminant with conductor $f_{0}$. The number of orbifold points of order 2 in $\W$ is the generalised class number $e_{2}(D)$ defined above.
\end{theorem}

Let us recall that the Prym image of any fibre of the \Xfamily{} $\X$ is given by $\PX=\C^{2}/\Lambda$, where $\Lambda=\PrymX\cdot\Z^{4}$ for
    \[\PrymX= \begin{pmatrix}
    -\frac{1+i}{2} & 1 & 1 & 0 \\
    1 & -(1+i) & 0 & 2 \\
    \end{pmatrix}\]
and that we have $\End_{\Q}\PX\cong M_2(\Q[\i])$.

We will first study the possible embeddings of $\O_{D}$ in $\End\PX$ as self-adjoint endomorphisms.

\begin{lemma}\label{lem:sqrtDforb2}Let $A$ be an element of $\End\PX$. The following are equivalent:
    \begin{enumerate}[label=\emph{(\roman*)}]
    \item $A$ is a self-adjoint endomorphism such that $A^{2}=D$;
    \item $A\coloneqq A_{\sqrt{D}}(a,b,c)=a\cdot A_{\gamma}+b\cdot A_{\delta}+c\i\cdot A_{\gamma\delta}$ for some $a,b,c\in\Z$ such that $a^{2}+b^{2}+c^{2}=D$.
    \end{enumerate}
\end{lemma}

\begin{proof}
By \autoref{EndPX}, any element of $\End_{\Q}\PX$ can be written as $A=a\cdot A_{\gamma}+b\cdot A_{\delta}+c\cdot A_{\gamma\delta}+ d\cdot\mathrm{Id}$, with $a,b,c,d\in\Q[\i]$. By~\autoref{eq:C4involutionM2C} it is clear that $A$ is self-adjoint if and only if $a,b,d\in\Q$ and $c\in\Q\cdot\i$. On the other hand, only scalars or pure quaternions satisfy $A^{2}\in\Q$, hence~$d=0$. A simple calculation shows that this implies $D=A^{2}=a^{2}+b^{2}+c^{2}$.

Now, one can check that the rational representation of such an element is given by
    \[R_{\sqrt{D}}(a,b,c)=\begin{pmatrix}
    a+b+c & -2c & 0 & 2a+2c \\
    a & -a-b-c & -a-c & 0 \\
    0 & 2b & a+b+c & 2a \\
    -b & 0 & -c & -a-b-c
    \end{pmatrix}\,,\]
therefore $A$ induces an endomorphism if and only if~$a,b,c\in \Z$.
\end{proof}

The analytic representation
    \[A_{\sqrt{D}}(a,b,c)=\begin{pmatrix} b & a\cdot\dfrac{1-\i}{2}-c\cdot\dfrac{1+\i}{2} \\ a(1+\i)-c(1-\i) & -b \end{pmatrix}\]
has eigenvectors
    \begin{equation}\label{eq:eigenvectors}
    \omega(a,b,c)^+=\begin{pmatrix}\dfrac{-1+\i}{2}\cdot\dfrac{a-c\,\i}{b+\sqrt{D}} \\[10pt] 1 \end{pmatrix}
    \mbox{ and\ \  }
    \omega(a,b,c)^-=\begin{pmatrix}\dfrac{-1+\i}{2}\cdot\dfrac{a-c\,\i}{b-\sqrt{D}} \\[10pt] 1 \end{pmatrix}\,.
    \end{equation}

The eigenvectors (almost) determine the triple $(a,b,c)$ and the discriminant~$D$.

\begin{lemma}\label{lem:sameeigen2}$A_{\sqrt{D}}(a,b,c)$ and $A_{\sqrt{D'}}(a',b',c')$ have the same eigenvectors if and only if
\begin{enumerate}[label=\emph{(\roman*)}]
    \item $D=m^{2}E$ and $D'=m'^{2}E$ for some discriminant $E$, with $\gcd(m,m')=1$, and
    \item Both $(a,b,c)$ and $(a',b',c')$ are integral multiples of a triple $(a_{0},b_{0},c_{0})\in\Z^{3}$ with~$a_{0}^{2}+b_{0}^{2}+c_{0}^{2}=D_{0}$.
\end{enumerate}
In particular, $A_{\sqrt{D}}(a,b,c)$ and $A_{\sqrt{D}}(a',b',c')$ have the same eigenvectors if and only if~$(a',b',c')=\pm(a,b,c)$. More precisely: $\omega(a,b,c)^+=\omega(-a,-b,-c)^-$ and $\omega(a,b,c)^-=\omega(-a,-b,-c)^+$.
\end{lemma}

\begin{proof}Suppose $A_{\sqrt{D}}(a,b,c)$ and $A_{\sqrt{D'}}(a',b',c')$ have the same eigenvectors, so that
    \[\dfrac{a-c\,\i}{b+\sqrt{D}} = \dfrac{a'-c'\,\i}{b'\pm \sqrt{D'}}\,.\]

This immediately implies that there has to be some discriminant $E$ such that $D=m^{2}E$ and $D'=m'^{2}E$, where we choose~$\gcd(m,m')=1$.

The equality above is equivalent to
    \begin{align*}
    ab' \pm am'\sqrt{E} & = a'b+a'm\sqrt{E} \\
    cb' \pm cm'\sqrt{E} & = c'b+c'm\sqrt{E} \,.
    \end{align*}
Since $E$ is not a square, this means $am' = \pm a'm$, $ab' = a'b$, $cm' = \pm c'm$ and $cb' = c'b$. Since $m$ and $m'$ are coprime we have
    \begin{align*}
    a &= ma_{0}\,, & b&=mb_{0}\,, & c&=mc_{0}\,,\quad\mbox{and} \\
    a' &= \pm m'a_{0}\,, & b'&= \pm m'b_{0}\,, & c'&= \pm m'c_{0}\,.
    \end{align*}
for some triple $(a_{0},b_{0},c_{0})\in\Z^{3}$. Dividing both sides of $a^{2}+b^{2}+c^{2}=D$ by $m^{2}$, we obtain $a_{0}^{2}+b_{0}^{2}+c_{0}^{2}=E$.

The converse is immediate.
\end{proof}

\begin{lemma}\label{lem:0mod4} Suppose $\PX$ admits real multiplication by $\O_{D}$. Then~$D\equiv 0\bmod{4}$.

Moreover, there is a bijection between the choices of real multiplication $\O_{D}\hookrightarrow \End\PX$ and the choices of triples $(a,b,c)$ as in~\autoref{lem:sqrtDforb2}.
\end{lemma}

\begin{proof}Let $\O_{D}\hookrightarrow \End\PX$ be a choice of real multiplication. The rational representation $R_T$ of the element $T\in\O_{D}$ will be given by $R_{\sqrt{D}}(a,b,c)/2$ or $(\mathrm{Id}+R_{\sqrt{D}}(a,b,c))/2$ for some $(a,b,c)$ satisfying the conditions of~\autoref{lem:sqrtDforb2}, depending on whether $D\equiv 0$ or $1\bmod{4}$ respectively. Therefore
    \[R_{T}(a,b,c)=\left\{\begin{array}{ll}
    \begin{pmatrix}
    \frac{a+b+c}{2} & -c & 0 & a+c \\
    \frac{a}{2} & -\frac{a+b+c}{2} & -\frac{a+c}{2} & 0 \\
    0 & b & \frac{a+b+c}{2} & a \\
    -\frac{b}{2} & 0 & -\frac{c}{2} & -\frac{a+b+c}{2}
    \end{pmatrix}
    \,, & \mbox{if $D\equiv 0\bmod{4}$},\\
    & \\
    \begin{pmatrix}
    \frac{1+a+b+c}{2} & -c & 0 & a+c \\
    \frac{a}{2} & \frac{1-a-b-c}{2} & -\frac{a+c}{2} & 0 \\
    0 & b & \frac{1+a+b+c}{2} & a \\
    -\frac{b}{2} & 0 & -\frac{c}{2} & \frac{1-a-b-c}{2}
    \end{pmatrix}
    \,, & \mbox{if $D\equiv 1\bmod{4}$}.\\
    \end{array}
    \right.\]

A simple parity check shows that $R_{T}(a,b,c)$ is always integral for $D\equiv 0\bmod{4}$ and never integral for~$D\equiv 1\bmod{4}$.

Conversely, every choice of $(a,b,c)$ gives a different embedding $\O_{D}\hookrightarrow \End\PX$ by~\autoref{lem:sameeigen2}.
\end{proof}

\begin{lemma}\label{lem:RMforC4} Let $D\equiv 0\bmod{4}$ be a discriminant with conductor $f_{0}$. A form $\omega$ is an eigenform for real multiplication by $\O_{D}$ if and only if it is the eigenform of some $A_{\sqrt{D}}(a,b,c)$ with~$\gcd(a,b,c,f_{0})=1$.
\end{lemma}

\begin{proof}By the previous lemma, any choice of real multiplication corresponds to a triple $(a,b,c)\in\Z^{3}$ as in~\autoref{lem:sqrtDforb2}.

By~\autoref{lem:sameeigen2}, such an embedding $\O_{D}\hookrightarrow \End\PX$, $T\mapsto A_{D}(a,b,c)\coloneqq A_{\sqrt{D}}(a,b,c)/2$ is proper if and only if~$\gcd(a,b,c,f_{0})=1$.
\end{proof}

\begin{proof}[Proof of~\autoref{thm:orbpt2and4}]By~\autoref{lem:RMforC4}, the set $\Hcal_{2}(D)$ counts choices of proper real multiplication $\O_{D}\hookrightarrow \End\PX$. Since every tuple $(a,b,c)\in\Hcal_2(D)$ gives two eigenforms and, by \autoref{lem:sameeigen2}, $(a,b,c)$ and $(-a,-b,-c)$ give the same eigenforms, there are exactly $|\Hcal_{2}(D)|$ eigenforms for real multiplication in $\PX$ for each $D\equiv 0\bmod{4}$, up to scaling. By~\cite[Prop. 4.6]{moellerprym}, each of them corresponds precisely to one element in some~$\P\Omega(\X_{t})^{-}(4)$. Recall also that, for each $t\in\P^*$, the isomorphism induced by the matrix $Q_{t}$, defined in~\autoref{eq:Qt}, allows us to see the four differentials of $\X_{t}$ given by~\autoref{lem:4foldforms} in the basis of differentials associated to~$\PrymX$.

In the case $D=8$, one has \[|\Hcal_{2}(8)|=|\{(\pm2,\pm2,0),(\pm2,0,\pm2),(0,\pm2,\pm2)\}|=12.\]
Using $Q_t$, it is easy to see that the eigenforms associated to the elements of $\Hcal_{2}(8)$ correspond to the elements of $\P\Omega(\X_{2})^{-}(4)$. More precisely, these eigenforms coincide, up to scaling, with the images $Q_{t}(\omegaX_{1})$, $Q_{t}(\omegaX_{2})$, $Q_{t}(-\omegaX_{1}+\omegaX_{2})$ and $Q_{t}(-t\omegaX_{1}+\omegaX_{2})$, for $t=-1,\nicefrac{1}{2},2$  (recall that $\X_{2}\cong\X_{-1}\cong\X_{\nicefrac{1}{2}}$). For example, by~\autoref{eq:eigenvectors} the matrix $A_{\sqrt{8}}(2,2,0)$ has as an eigenvector $\bigl(\frac{1-\i}{-2-\sqrt{8}},1\bigr)$, which is a multiple of $Q_{2}(\omegaX_{1})$.
As a consequence of \autoref{fermateigen}, the curve $\W[D=8]$ has one orbifold point of order 4 and no orbifold points of order 2. In particular, no other $\W$ can contain a point of order $4$.

Arguing the same way for $D=12$ and using \autoref{lem:actionbeta}, one finds the (unique) orbifold point of order 6 on~$\W[D=12]$ in accordance with \autoref{thm:orbpt3and6}.

Now, let $D\neq8,12$. By \autoref{thm:orbicrit}, we know that $\X_{t}\not\cong\X_{2}$.
As, by \autoref{lem:4foldforms}, for each $t\in\P^{*}$ the set $\P\Omega(\X_{t})^{-}(4)$ has four elements and the map $t\mapsto\X_t$ is generically $6:1$ (cf. \autoref{lem:families}), we have to divide $|\Hcal_{2}(D)|$ by $4\cdot 6=24$ to get the correct number of orbifold points.
\end{proof}

\subsection{Points of order 3 and 6}

\begin{theorem}\label{thm:orbpt3and6}The curve $\W[D=12]$ has one orbifold point of order 6. Moreover, no curve $\W$ has orbifold points of order 6.

Let $D\neq 12$ be a discriminant with conductor $f_{0}$. The number of orbifold points of order 3 in $\W$ is the generalised class number $e_{3}(D)$ defined above.
\end{theorem}

In the case of the \Yfamily{} $\Y$ we are, by \autoref{lem:4foldforms}, only interested in the case where $\omegaY_{2}$ is an eigenform for real multiplication. Using the bases constructed in Lemmas~\ref{lem:forms} and~\ref{lem:hombasis}, we get the following.

\begin{lemma}\label{lem:orbifoldC6} The curve $\Y_{t}$ is an orbifold point of $\W$ if and only if the matrix
    \[A_{D}\coloneqq
    \begin{pmatrix}T&0\\0&-T\\\end{pmatrix}\] 
is the analytic representation of an endomorphism of $\PY$ and $A_{D'}$ is not for all discriminants $D'$ dividing~$D$.

The orbifold order of $\Y_{t}$ is 6 if $\Y_{t}\cong\Y_{\nicefrac{1}{2}}$ and 3 otherwise.
\end{lemma}

\begin{proof}The form $\omega_{2}$ is an eigenform for real multiplication by $\O_{D}$ on $\PY$ if and only if there is a matrix $\left(\begin{smallmatrix}T&0\\ \gamma &-T\\\end{smallmatrix}\right)$ for some $\gamma \in\C$ representing a self-adjoint endomorphism of $\PY$ and, moreover, the corresponding action of $\O_{D}$ is proper. By the explicit description of the Rosati involution in this basis~\autoref{eq:C6involutionM2C}, the self-adjoint condition implies $\gamma =0$. Moreover, the action of $\O_{D}$ is proper if and only if $A_{D'}$ does not induce an endomorphism for every discriminant~$D'|D$.

The claim about the orbifold order follows from \autoref{thm:families} and~\autoref{thm:orbicrit}.
\end{proof}

Using the period matrix $\PY$ we can compute the rational representation $R_{D}$ for such an $A_{D}$ in terms of $f$ and find conditions for $R_{D}$ to be integral. Remember that the parameter $f=f(t)$ lives in the disc of radius~$\nicefrac{1}{\sqrt{2}}$.

\begin{prop}\label{prop:conditionorbpt3}Let $f\in\C$ such that $|f|^2<\nicefrac{1}{2}$ and let $\PY$ be as above. The matrix $A_{D}$ induces a self-adjoint endomorphism of the corresponding Prym variety if and only if there exist integers $a,b,c\in\Z$ such that
    \begin{enumerate}[label=\emph{(\roman*)}]
    \item $2a^{2}-3b^{2}-c^{2}=2D$, and
    \item $f=f(a,b,c,D)\coloneqq \dfrac{\sqrt{3}b\i + c}{2(a-\sqrt{D})}$.
    \end{enumerate}
\end{prop}

\begin{figure}[!htp]
\centering
\includegraphics{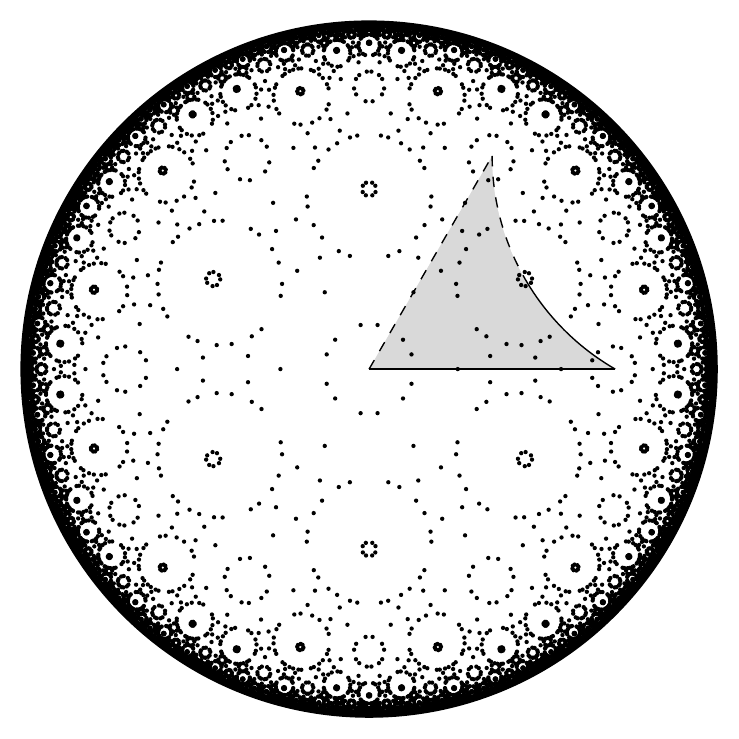}
\caption{Points in the disc of radius $\nicefrac{1}{\sqrt{2}}$ satisfying the conditions of~\autoref{prop:conditionorbpt3} for $D=3257$ together with the fundamental domain of $\Delta(2,6,6)$.}
\label{fig:C6points3257}
\end{figure}

\begin{proof}Given an element of $\End_{\Q}\PY$ with analytic representation $A$, its rational representation $R$ is given by
    \[R= \begin{pmatrix}\PY\\ \overline{\PY}\\\end{pmatrix}^{-1}
    \begin{pmatrix}A & 0 \\ 0 & \overline{A} \\\end{pmatrix}
    \begin{pmatrix}\PY\\ \overline{\PY}\\\end{pmatrix}\,.
    \]

Suppose that $A_{D}$ induces a self-adjoint endomorphism. In particular, the matrix $A_{\sqrt{D}}=\left(\begin{smallmatrix}\sqrt{D}&0\\0&-\sqrt{D}\\\end{smallmatrix}\right)$ also induces an endomorphism and a tedious but straightforward calculation shows that the corresponding rational representation is
    \[R_{\sqrt{D}}=\begin{pmatrix}
    B_{1} & 0 & B_{3} & B_{2} \\
    0 & B_{1} & B_{2} & B_{4} \\
    2B_{4} & -2B_{2} & -B_{1} & 0 \\
    -2B_{2} & 2B_{3} & 0 & -B_{1} \\
    \end{pmatrix}\,,\]
where
    \begin{align*}
    B_{1} &=\frac{\sqrt{D}(2|f|^{2}+1)}{2|f|^{2}-1}\,,\\
    B_{2} &=-\dfrac{2\sqrt{3}\sqrt{D}(|f|^{2}-f^{2})\i}{3f(2|f|^{2}-1)}\,, \\
    B_{3} &= \dfrac{\sqrt{3}\sqrt{D}(|f|^{2}-f^{2})\i}{3f(2|f|^{2}-1)} + \dfrac{\sqrt{D}(|f|^{2}+f^{2})}{f(2|f|^{2}-1)}\quad\text{and}\\
    B_{4} &= \dfrac{\sqrt{3}\sqrt{D}(|f|^{2}-f^{2})\i}{3f(2|f|^{2}-1)} - \dfrac{\sqrt{D}(|f|^{2}+f^{2})}{f(2|f|^{2}-1)}.
    \end{align*}

We define $a\coloneqq B_{1}\in\Z$ and from the expression above we get that
    \begin{equation}\label{eq:normf^2}
    |f|^{2} =\frac{1}{2}\cdot\frac{a+\sqrt{D}}{a-\sqrt{D}}\,.
    \end{equation}

Moreover, since $|f|^{2}-f^{2}=-2\i\cdot f\Im{f}$, $|f|^{2}+f^{2}=2f\Re{f}$ and $2|f|^{2}-1=\frac{2\sqrt{D}}{a-\sqrt{D}}$, the expressions above imply
    \[b \coloneqq B_{2} = \dfrac{2(a-\sqrt{D})\Im(f)}{\sqrt{3}}\quad\mbox{and}\quad
    c \coloneqq 2B_{3}-B_{2} = -2B_{4}+B_{2} = 2(a-\sqrt{D})\Re(f)\,,\]
so that
    \[f = \frac{c+\sqrt{3}b\i}{2(a-\sqrt{D})}\,,\]
and~\autoref{eq:normf^2} implies that~$2a^{2}-3b^{2}-c^{2}=2D$, as claimed.

\medskip

Conversely, suppose that $a,b,c\in\Z$ satisfy the conditions of the proposition and define $f=f(a,b,c,D)$ as above. The rational representation of $A_{D}$ (at the point corresponding to $f$) is given by $R_{T}=R_{\sqrt{D}}/2$ or $(\mathrm{Id}+R_{\sqrt{D}})/2$, depending on whether $D\equiv 0$ or $1\bmod{4}$, respectively, and therefore
    \[R_{T}=\left\{\begin{array}{ll}
    \begin{pmatrix}
    \dfrac{a}{2} & 0 & \dfrac{b+c}{2} & b \\
    0 & \dfrac{a}{2} & b & \dfrac{b-c}{2} \\
    b-c & -2b & -\dfrac{a}{2} & 0 \\
    -2b & b+c & 0 & -\dfrac{a}{2} \\
    \end{pmatrix}
    \,, & \mbox{if $D\equiv 0\bmod{4}$},\\
    & \\
    \begin{pmatrix}
    \dfrac{1+a}{2} & 0 & \dfrac{b+c}{2} & b \\
    0 & \dfrac{1+a}{2} & b & \dfrac{b-c}{2} \\
    b-c & -2b & \dfrac{1-a}{2} & 0 \\
    -2b & b+c & 0 & \dfrac{1-a}{2} \\
    \end{pmatrix}
    \,, & \mbox{if $D\equiv 1\bmod{4}$}.\\
    \end{array}
    \right.\]

Considering the equality $2a^{2}-3b^{2}-c^{2}\equiv 2D\bmod{8}$, one sees that
\begin{itemize}
    \item $a$, $b$ and $c$ are even if $D\equiv 0\bmod{4}$, and
    \item $a$ is odd and $b$ and $c$ are even if $D\equiv 1\bmod{4}$
\end{itemize}
and therefore~$R_{T}\in M_{4}(\Z)$ in both cases.
\end{proof}

To compute the number of orbifold points on $\W$, we now count, for each discriminant $D$, how many points $f(a,b,c,D)$ in the fundamental domain of $\Delta(2,6,6)$ satisfy the previous conditions. Recall from~\autoref{subsec:PrymC6} that we consider the fundamental domain for the triangle group $\Delta(2,6,6)$ depicted in~\autoref{fig:trianglegroup}.

\begin{lemma}\label{lem:ptsinfunddomain}Let $\widetilde{\Hcal}_{3}(D)$ be the set of triples of integers $(a,b,c)$ such that
    \begin{enumerate}[label=\emph{(\roman*)}]
        \item\label{enum:H6_square} $2a^{2}-3b^{2}-c^{2}=2D$;
        \item\label{enum:H6_norm} $-3\sqrt{D}<a<-\sqrt{D}$;
        \item\label{enum:H6_angle} $c < b \le 0$;  
        \item\label{enum:H6_disc} Either $4a - 3b - 3c < 0$, or $4a - 3b - 3c = 0$ and $c < 3b$. 
    \end{enumerate}
The set $\widetilde{\Hcal}_{3}(D)$ agrees with the triples $(a,b,c)$ in~\autoref{prop:conditionorbpt3} that yield a point $f(a,b,c,D)$ in the fundamental domain of~$\Delta(2,6,6)$.
\end{lemma}

\begin{rem}Note that $\widetilde{\Hcal}_{3}(D)$ agrees with the set $\Hcal_{3}(D)$ defined above except for the condition on the $\gcd$. This condition will ensure that the embedding of $\O_{D}$ into $\End\PY$ is proper.
\end{rem}

\begin{proof} Recall that we are using the fundamental domain depicted in~\autoref{fig:trianglegroup}, whose vertices have been calculated in~\autoref{prop:trianglegroup}. Condition~\ref{enum:H6_norm} ensures that $0\le |f|^{2} \le \nicefrac{1}{4}$ and condition~\ref{enum:H6_angle} that $0\le \arg{f} < \nicefrac{\pi}{3}$. Now, the geodesic joining $f(0)$ and $f(1)$ is an arc of circumference $|z-(3+\sqrt{3}\i)/4|^{2}=\nicefrac{1}{4}$. Therefore, $f$ lives on the (open) half-disc containing the origin, determined by this geodesic, if and only if
	\[ \left|f-\frac{3+\sqrt{3}\i}{4}\right|^{2} = \left(\frac{c}{2(a-\sqrt{D})}-\frac{3}{4}\right)^2+\left(\frac{\sqrt{3}b}{2(a-\sqrt{D})}-\frac{\sqrt{3}}{4}\right)^2
	\ge \frac{1}{4}\,. \]
Expanding this expression and using the previous conditions, one gets the first part of condition~\ref{enum:H6_disc}. Since the sides joining $f(1)$ and $f(\nicefrac{1}{2})$, and $f(\nicefrac{1}{2})$ and $f(0)$ are identified by an element of order 2 in~$\Delta(2,6,6)$, we need to count only the points $f$ that lie on one of them, say the arc of the geodesic joining $f(1)$ and $f(\nicefrac{1}{2})$. Proceeding as before, we obtain the second part of condition~\ref{enum:H6_disc}.
\end{proof}

\begin{proof}[Proof of~\autoref{thm:orbpt3and6}]First note that if $D=g^{2}D'$, then
    \begin{equation}\label{eq:finC6}f(a,b,c,D)=f(a',b',c',D')\quad\mbox{ if and only if }a=ga',\ b=gb' \mbox{ and }c=gc'\,.
    \end{equation}

Since 12 is a fundamental discriminant,~\autoref{lem:orbifoldC6} and~\autoref{lem:ptsinfunddomain} imply that $\W[D=12]$ has one orbifold point of order 6. Moreover, this is the only curve with an orbifold point of order 6 because, by~\autoref{eq:finC6} above, the point $f(a,b,c,D)$ can only correspond to $t=\nicefrac{1}{2}$ if one has $D=f_{0}^{2}D_{0}$ for~$D_{0}=12$.

Now let $D\neq12$. By~\autoref{lem:orbifoldC6} and~\autoref{lem:ptsinfunddomain}, we only need to prove that $\Hcal_{3}(D)$ is the set of triples in $\widetilde{\Hcal}$ which are not contained in any $\widetilde{\Hcal}_{3}(D')$, for discriminants~$D'|D$. This is true since, by~\autoref{eq:finC6}, $(a,b,c)\in \widetilde{\Hcal}_{3}(D)$ is not contained in any $\widetilde{\Hcal}_{3}(D')$ if and only if~$\gcd(a,b,c,f_{0})=1$.
\end{proof}

\section{Examples}\label{sec:examples}

\noindent{\bf Example 1} ($\W[D=12]$ and $\W[D=20]$). 
The curve $\W[D=12]$ has genus zero, two cusps and one orbifold point of order 6, and the curve $\W[D=20]$ has genus zero, four cusps and one elliptic point of order 2, cf. \cite[Ex. 4.4]{moellerprym}. Our results agree with this. These are the curves $V(S_{1})$ and $V(S_{2})$ in~\cite{mcmprym}.

\medskip

\noindent{\bf Example 2} ($\W[D=8]$). By \autoref{thm:orbpt2and4} and \autoref{thm:orbpt3and6}, we find that $\W[D=8]$ has one orbifold point of order 3 and one orbifold point of order 4. By~\cite[Thm. C.1]{lanneaunguyen} the number of cusps is $C(\W[D=8])=1$, the curve is connected, and by~\cite[Thm. 0.2]{moellerprym} the Euler characteristic is $\chi(\W[D=8])=-5/12$. We can then use~\autoref{eq:invariants} to compute its genus as $g(\W[D=8])=0$.

\medskip

\noindent{\bf Example 3} ($\W[D=2828]$). \autoref{thm:orbpt2and4} and \autoref{thm:orbpt3and6} also tell us that $\W[D=2828]$ has six orbifold points of order 2. They correspond to the $|\Hcal_{2}(2828)|=144$ eigenforms for real multiplication by $\O_{2828}$ in $\PX$, as in~\autoref{eq:eigenvectors}, divided by 24. In \autoref{fig:example}, we depict the first coordinate of these eigenforms in the complex plane.

As for the orbifold points of order 3, there are twenty of them. They correspond to the twenty points on the Shimura curve isomorphic to $\D/\Delta(2,6,6)$ admitting proper real multiplication by $\O_{2828}$. In \autoref{fig:example}, we depict the preimage of these 20 points in $\D$, that is the points $f(a,b,c,2828)$ as in~\autoref{prop:conditionorbpt3}.

The number of cusps is $C(\W[D=2828])=68$, the curve is connected, and the Euler characteristic is $\chi(\W[D=2828])=-8245/3$. Therefore, by~\autoref{eq:invariants}, the genus is $g(\W[D=2828])=1333$.

\begin{figure}[!htp]
\centering
\begin{minipage}[b]{0.48\textwidth}
\includegraphics[width=\textwidth]{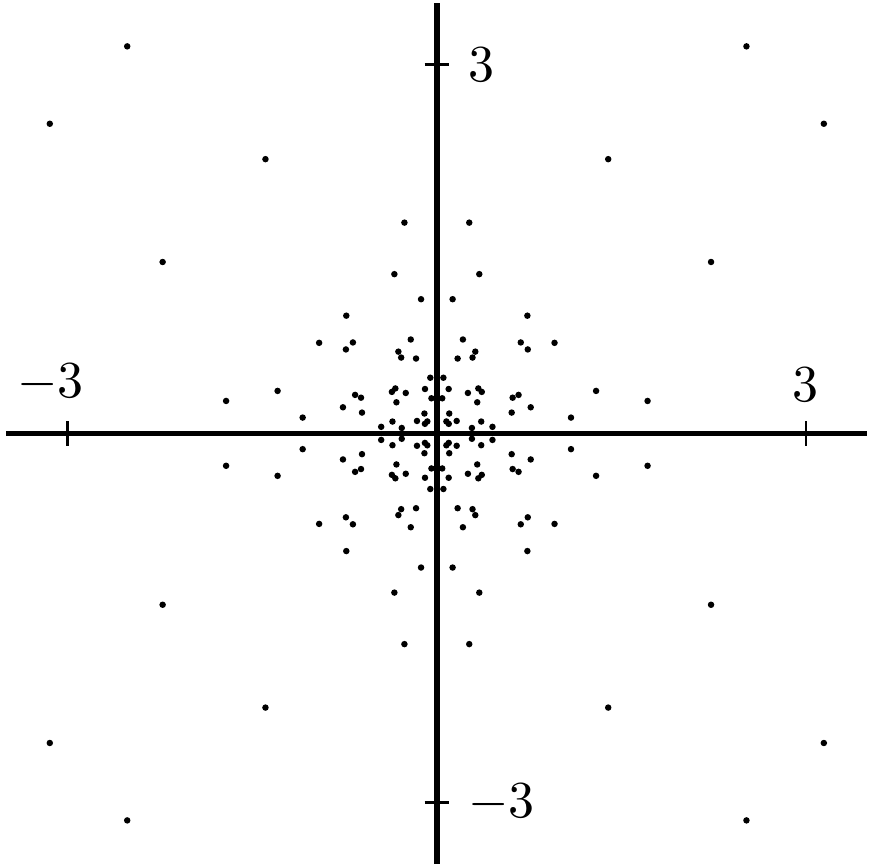}
\end{minipage}
\begin{minipage}[b]{0.48\textwidth}
\includegraphics[width=\textwidth]{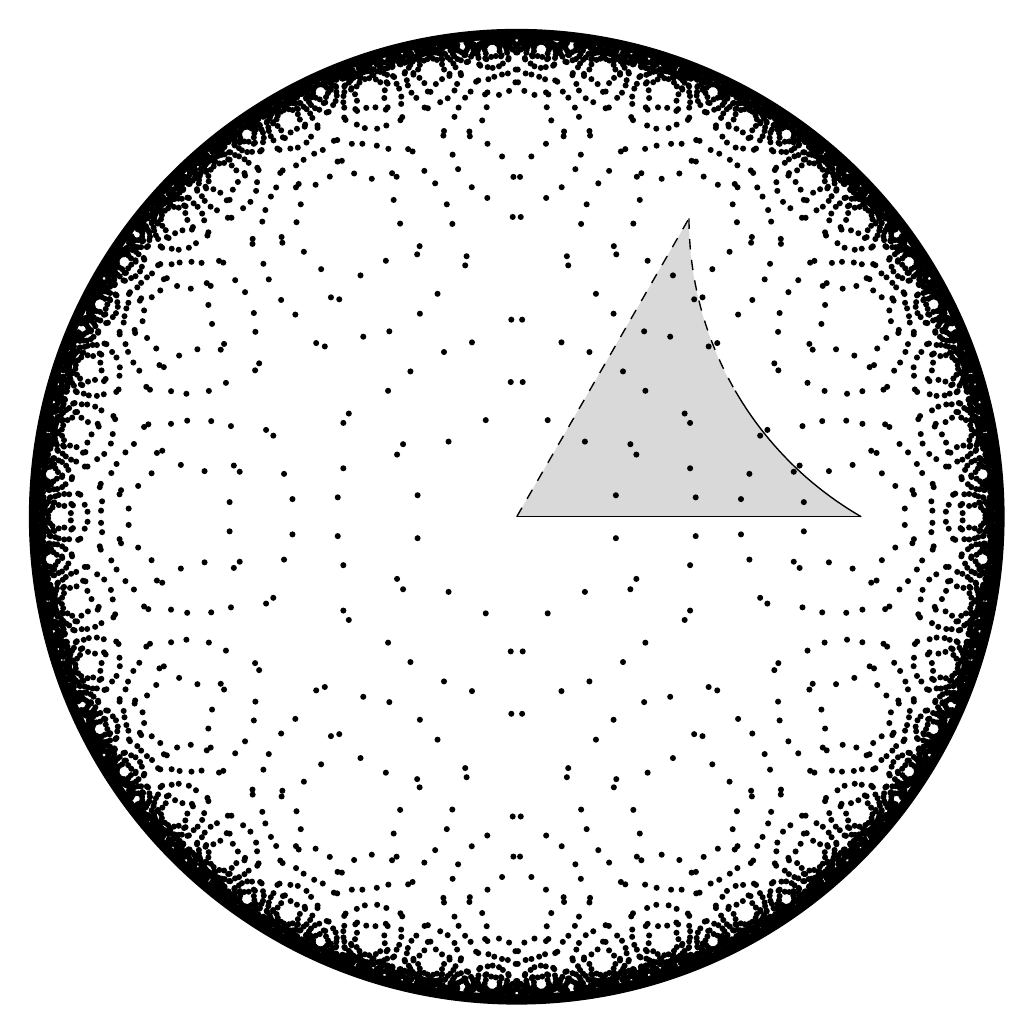}
\end{minipage}
\caption{Orbifold points of order 2 and 3 in $\W[D=2828]$.}\label{fig:example}
\end{figure}

\begin{table}[ht]
\centering
\begin{tabular}{r | r | r | r | r | r}
  $D$ & $\chi$ & $C$ & $g$ & $e_{2}$ & $e_{3}$  \\\hline
17&-5/3&3&0&0&1\\
20&-5/2&4&0&1&0\\
24&-5/2&4&0&1&0\\
28&-10/3&4&0&0&2\\
32&-5&7&0&0&0\\
33&-5&7&0&0&0\\
40&-35/6&6&0&1&2\\
41&-20/3&8&0&0&1\\
44&-35/6&6&0&1&2\\
48&-10&10&1&0&0\\
52&-25/2&12&1&1&0\\
56&-25/3&6&1&2&2\\
57&-35/3&11&1&0&1\\
60&-10&8&2&0&0\\
65&-40/3&12&1&0&2\\
68&-15&14&1&2&0\\
72&-25/2&10&2&1&0\\
73&-55/3&17&1&0&2\\
76&-95/6&14&1&1&2\\
80&-20&16&3&0&0\\
84&-25&16&5&2&0\\
88&-115/6&16&1&1&4\\
89&-65/3&15&4&0&1\\
92&-50/3&8&4&0&4\\
96&-30&20&6&0&0\\
97&-85/3&21&4&0&2\\
104&-125/6&10&5&3&2\\
105&-30&18&7&0&0\\
108&-45/2&14&5&1&0\\
112&-40&24&9&0&0\\
113&-30&18&6&0&3\\
116&-75/2&20&9&3&0\\
120&-85/3&12&8&2&2\\
124&-100/3&16&9&0&2\\
128&-40&22&10&0&0\\
129&-125/3&25&9&0&1\\
132&-45&30&8&2&0\\
136&-115/3&20&9&2&2\\
137&-40&22&9&0&3\\
140&-95/3&12&9&2&4\\
145&-160/3&32&11&0&2\\
148&-125/2&36&14&1&0\\
152&-205/6&12&10&3&4\\
153&-50&30&11&0&0\\
156&-130/3&16&14&0&2\\
160&-70&42&15&0&0\\
161&-160/3&22&16&0&2\\
\end{tabular}
\qquad
\begin{tabular}{r | r | r | r | r | r}
  $D$ & $\chi$ & $C$ & $g$ & $e_{2}$ & $e_{3}$  \\\hline
164&-60&32&14&4&0\\
168&-45&16&15&2&0\\
172&-105/2&22&14&1&6\\
176&-70&30&21&0&0\\
177&-65&31&18&0&0\\
180&-75&32&22&2&0\\
184&-185/3&22&19&2&4\\
185&-190/3&26&19&0&2\\
188&-140/3&12&17&0&4\\
192&-80&36&23&0&0\\
193&-245/3&39&21&0&4\\
200&-325/6&18&17&3&4\\
201&-245/3&37&23&0&1\\
204&-65&28&19&2&0\\
208&-100&48&27&0&0\\
209&-235/3&35&22&0&2\\
212&-175/2&28&30&3&0\\
216&-135/2&32&18&3&0\\
217&-290/3&42&27&0&4\\
220&-230/3&32&22&0&4\\
224&-100&34&34&0&0\\
228&-105&46&30&2&0\\
232&-165/2&30&25&1&6\\
233&-265/3&29&29&0&5\\
236&-425/6&26&22&3&2\\
240&-120&40&41&0&0\\
241&-355/3&49&35&0&2\\
244&-275/2&52&43&3&0\\
248&-70&14&26&4&6\\
249&-115&45&36&0&0\\
252&-80&24&29&0&0\\
257&-100&34&33&0&3\\
260&-120&48&36&4&0\\
264&-280/3&32&30&4&2\\
265&-400/3&56&39&0&2\\
268&-205/2&30&35&1&6\\
272&-120&44&39&0&0\\
273&-370/3&38&43&0&2\\
276&-150&40&55&4&0\\
280&-335/3&36&37&2&4\\
281&-125&45&40&0&3\\
284&-290/3&20&38&0&4\\
288&-150&54&49&0&0\\
292&-165&74&46&2&0\\
296&-205/2&22&38&5&6\\
297&-135&49&44&0&0\\
300&-325/3&28&40&2&2\\
\end{tabular}
\medskip
\caption{Topological invariants of the Prym-Teichmüller curves $\W$ for $D$ up to \maxD{}. For $D\equiv 1\mod 8$, we give the homeomorphism type of one of the two homeomorphic components, cf. \cite{components}.}\label{tab:thetable}
\end{table}

\section{Flat geometry of orbifold points}\label{sec:flat}

In this section we will briefly describe the translation surfaces corresponding to the \Yfamily{} and to the \Xfamily.

\medskip

Recall that, by~\autoref{lem:4foldforms}, the general member $\Y_{t}$ of the \Yfamily{} has only one differential with a single zero, namely $\omegaY_{2}$. Flat surfaces $(\Y_{t},\omegaY_{2})$ arise from the following \emph{double windmill} construction, which also explains the name \Yfamily: for each period $\tau\in\mathbb{C}$ consider the \enquote{blade} depicted on the left side of~\autoref{fig:windmill}, where $\overrightarrow{AF}=\tau$, $|AF|=|EF|$, $|AB|=|BC|$ and $|CD|=|DE|$. We normalise the differential by fixing the edge $\overrightarrow{AB}$ to be $\i$. Now take 6 copies of the blade and glue them together with side pairing as in the right side of the picture. One can check that this yields a genus 3 curve and that the corresponding differential has a unique zero, namely the black point in the picture. Moreover, there is an obvious order 6 automorphism $\alpha$ of the curve, induced by the composition of a rotation of order three on each of the two windmills and a rotation of order two of the whole picture around the white point on the common side of the two windmills. This automorphism fixes the black point and exchanges cyclically the three white points, the two centres of the windmills and the two crossed points, respectively.

It is again easy to check that $\alpha^{3}$ corresponds to the Prym involution. Therefore, the corresponding curve belongs to the \Yfamily{}. The black point corresponds to the preimage of~$\infty$ under the cyclic cover $\Y_{t}\rightarrow\P^{1}$, the three white points correspond to the preimages of $t$, the two crossed points to the preimages of 1 and the centres of the windmills to the two preimages of 0.

\begin{figure}
\begin{minipage}{\textwidth}
  \centering
  \raisebox{-0.5\height}{\includegraphics[width=0.23\textwidth]{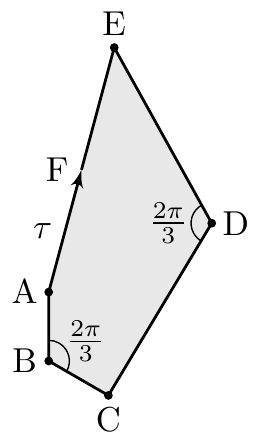}}
  \hspace*{0.01\textwidth}
  \raisebox{-0.5\height}{\includegraphics[width=0.65\textwidth]{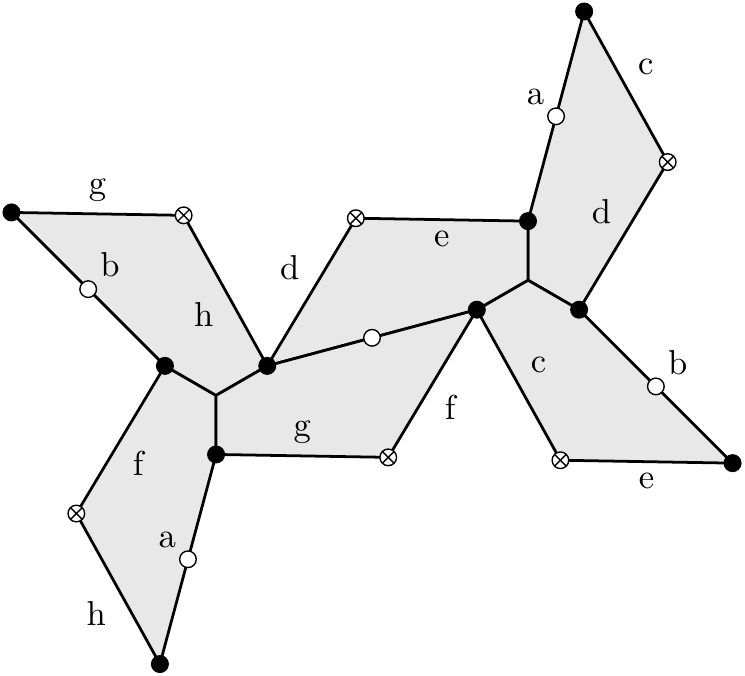}}
\end{minipage}
\caption{Double windmill for the period $\tau\in\mathbb{C}$.}\label{fig:windmill}
\end{figure}

\

\noindent{\bf Example 4}. The special point $\Y_{\nicefrac{1}{2}}$ has an extra automorphism $\beta$ of order 12. The corresponding flat surface is depicted below in~\autoref{fig:windmill_special}. The automorphism $\beta$ corresponds to first rotating each of the blades by $\pi/2$ around each of the white points and reglueing, and then composing with $\alpha$.

\

\noindent{\bf Example 5}. Each component of the Prym-Teichmüller curve $\W[D=17]$ has one orbifold point of order $3$ (cf. \cite{components}). Using the lengths described in~\cite{lanneaunguyen} and Mukamel's implemented algorithm from~\cite{mukamelfundamental}, one finds that this orbifold point corresponds to the $S$-shaped table depicted in~\autoref{fig:STable_disc17}, where
\begin{align*}
	d &= \left( \frac{11\sqrt{17} - 35}{52}, -\sqrt{3}\cdot\frac{\left(17\sqrt{17} - 73\right)}{52}\right)\,, \\
	a &= \left(\frac{\sqrt{17} -1}{2}, 0 \right)\,, \\
	b = e &= \left(\frac{-\sqrt{17} + 5}{2}, 0 \right)\,, \\
	f = c &= \left( \frac{-3\sqrt{17} - 33}{52}, \sqrt{3}\cdot\frac{\left(7\sqrt{17} - 27\right)}{52}\right)\,.
\end{align*}	

\begin{figure}
\centering
\includegraphics[scale=0.8]{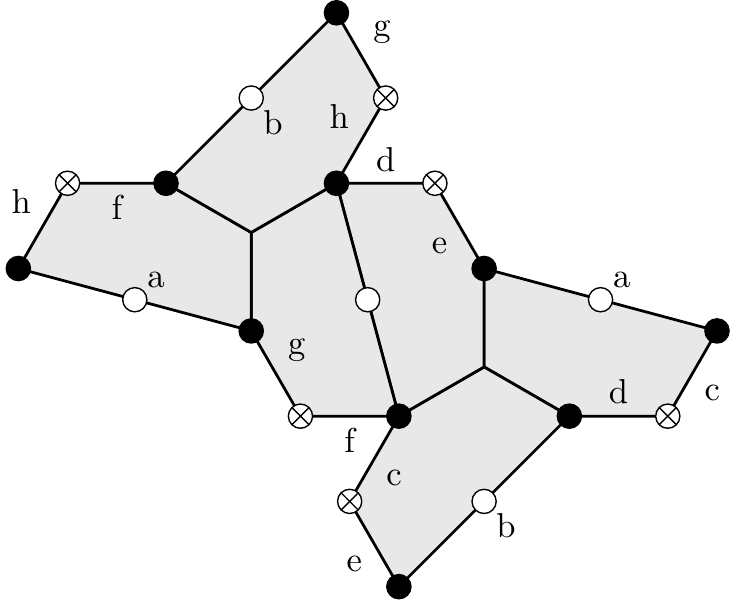}
\caption{Double windmill corresponding to the special point $(\Y_{\nicefrac{1}{2}},\omegaY_{2})$.}
\label{fig:windmill_special}
\end{figure}

\begin{figure}
\centering
\includegraphics{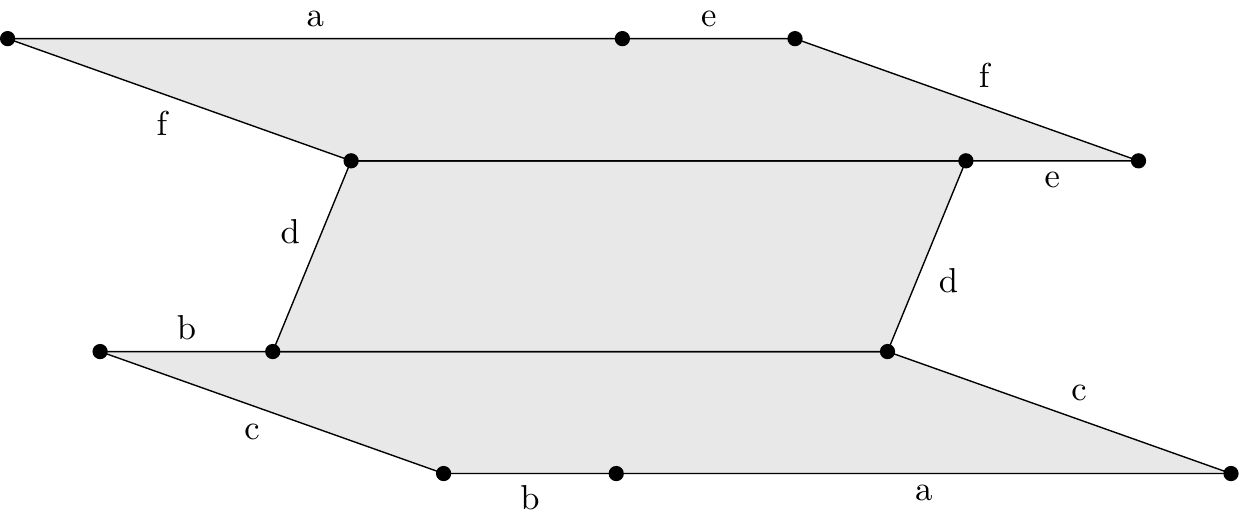}
\caption{$S$-shaped table for the orbifold point of order 3 on $\W[D=17]$ (Y axis scaled by a factor of 5).}
\label{fig:STable_disc17}
\end{figure}

Since the automorphism $\alpha$ of order 6 fixes the zero of the differential and interchanges cyclically the preimages of 0, 1 and $t$ respectively, one can easily detect these points. The three preimages of $t$ are, together with the preimage of~$\infty$, the fixed points of the Prym involution, which is just a rotation of the whole picture through an angle of $\pi$. Therefore, they correspond to the centre of the $S$-table and to the midpoints of edges $a$ and $d$.

As for the preimages of 0 and 1, they can be found as the fixed points of $\alpha^{2}$. Since the angle around the zero of the differential is $10\pi$, the automorphism $\alpha$ corresponds to a rotation of angle $10\pi/6$ around that point.

Cutting appropriately the $S$-shaped table into pieces and reglueing them yields the double windmill in~\autoref{fig:windmill_disc17}. Note that in this case the differential is not normalised in the same way as in our construction.

\begin{figure}
\centering
\includegraphics{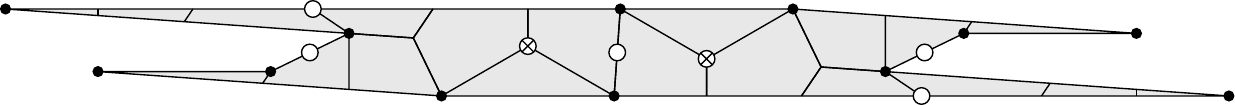}

\includegraphics{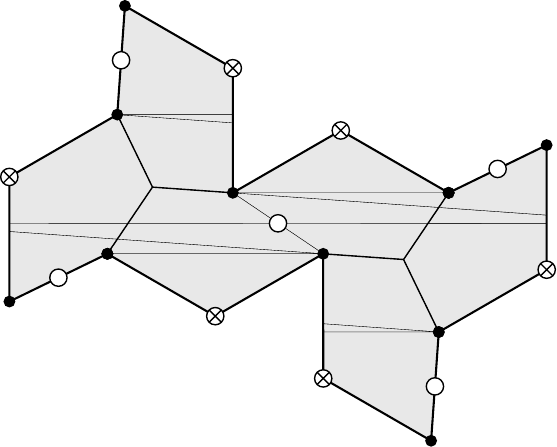}
\caption{Double windmill cut and pasted from the $S$-shaped table corresponding to the orbifold point of order three on $\W[D=17]$.}
\label{fig:windmill_disc17}
\end{figure}

\

One can similarly construct the flat surfaces associated to the \Xfamily{} via the following \emph{four-leaf clover} construction, which is again responsible for the name. Let us consider the differential $\omegaX_{1}$ in $\Omega(\X_{t})$, which by~\autoref{lem:4foldforms} has a zero at the preimage of~$\infty$. Flat surfaces $(\X_{t},\omegaX_{1})$ can be constructed in the following way: for each period $\tau\in\mathbb{C}$ we consider the \enquote{blade} on the left side of~\autoref{fig:wollmilchsau}, where $\overrightarrow{AF}=\tau$, $|AF|=|EF|$, $|AB|=|BC|$ and $|CD|=|DE|$. We again normalise the differential by fixing the edge $\overrightarrow{AB}$ to be~$\i$. Now we glue 4 copies of the blade with side pairings as in the right side of the picture. Again, this yields a genus 3 curve together with an abelian differential with a single zero, namely the black point in the picture. The order 4 automorphism $\alpha$ induced by a rotation of order four around the centre of the windmill fixes four points: the centre, the black point, the white point and the crossed point. The square $\alpha^{2}$ corresponds to the Prym involution, and therefore the corresponding curve belongs to the \Xfamily{}. In our construction, the black point corresponds to the preimage of~$\infty$ under the cyclic cover $\X_{t}\rightarrow\P^{1}$, the white point to the preimage of $t$, the crossed point to the preimage of 1 and the centre of the windmill to the preimage of 0.

\begin{figure}
\begin{minipage}{\textwidth}
  \centering
  \raisebox{-0.5\height}{\includegraphics[width=0.25\textwidth]{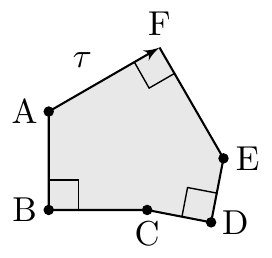}}
  \hspace*{0.05\textwidth}
  \raisebox{-0.5\height}{\includegraphics[width=0.45\textwidth]{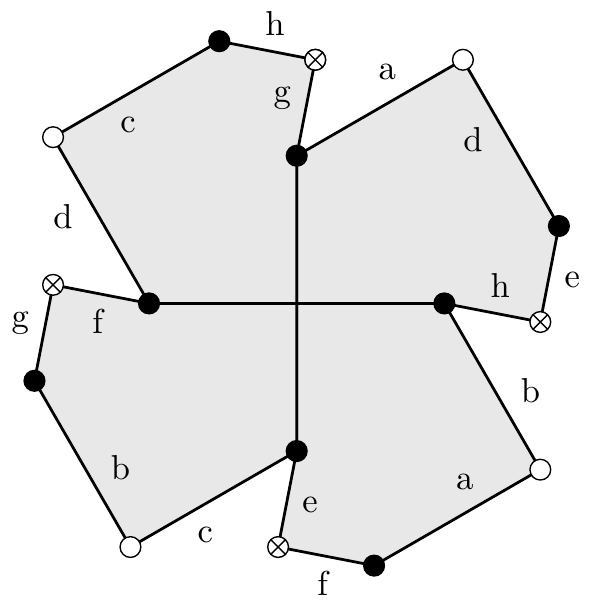}}
\end{minipage}
\caption{Flat surface corresponding to $(\X_{t},\omegaX_{1})$ for the period $\tau\in\mathbb{C}$.}\label{fig:wollmilchsau}
\end{figure}

\FloatBarrier

\emergencystretch=3em
\printbibliography

\end{document}